\documentclass[12pt,a4paper]{amsart}

\usepackage{amssymb}
\usepackage{amsmath,amsthm,geometry,enumerate}
\usepackage{graphicx}
\usepackage{psfrag}
\usepackage{amscd}
\usepackage{comment}
\usepackage{hyperref}
\usepackage[all]{xy}
\usepackage{tikz-cd}
\usepackage{booktabs}
\usepackage{color}
\usepackage{todonotes}
\usepackage{braket}

\numberwithin{equation}{section} \makeatletter

\DeclareMathOperator{\stab}{\sigma}
\DeclareMathOperator{\tw}{Tw}
\DeclareMathOperator{\pic}{Pic}

\DeclareMathOperator{\ShHom}{\mathscr{H}\text{\kern -3pt {\calligra\large om}}\,}

\DeclareMathOperator{\spec}{Spec}

\newcommand{\Z}{\mathbf{Z}}
\DeclareMathOperator{\Edges}{Edges}
\DeclareMathOperator{\Verts}{Vert}
\newcommand{\fcj}{\overline{J}}
\newcommand{\fcjs}{\overline{\mathcal J}}
\newcommand{\fcuj}[1][{g,n}]{\overline{\mathcal J}_{{#1}}}
\newcommand{\tGamma}{{\widetilde{\Gamma}}}

\newcommand{\Mb}{\overline{\mathcal{M}}}
\newcommand{\Cb}{\overline{\mathcal{C}}}

\newcommand{\Jb}{\overline{\mathcal{J}}}

\newcommand{\gammazero}{G}

\newcommand{\Pic}{\operatorname{Pic}}

\newcommand{\GSym}{\operatorname{GSym}}
\newcommand{\Spec}{ {\operatorname{Spec}}}
\newcommand{\Simp}{ {\operatorname{Simp}}}
\newcommand{\NF}{ {\operatorname{N}}}

\newtheorem{proposition}[equation]{Proposition}
\newtheorem{corollary}[equation]{Corollary}
\newtheorem{lemma}[equation]{Lemma}
\newtheorem{theorem}[equation]{Theorem}

\theoremstyle{definition}
\newtheorem{definition}[equation]{Definition}
\newtheorem{remark}[equation]{\textbf{Remark}}
\newtheorem{example}[equation]{\textbf{Example}}
\newtheorem{question}[equation]{\textbf{Question}}

\usepackage{color}
\definecolor{forestgreen}{rgb}{0.13, 0.55, 0.13}

\begin{document}

\title{Stability conditions for line bundles on  nodal curves}

\author{Nicola Pagani}

\address{N.~Pagani, Department of Mathematical Sciences, University of Liverpool, Liverpool, L69 7ZL, United Kingdom}
\email{pagani@liv.ac.uk}
\urladdr{http://pcwww.liv.ac.uk/~pagani/}

\author{Orsola Tommasi}

\address{O.~Tommasi, Dipartimento di Matematica ``Tullio Levi-Civita'', University of Padova, via Trieste 63, IT-35127 Padova, Italy}
\email{tommasi@math.unipd.it}
\urladdr{http://www.math.unipd.it/~tommasi}

\begin{abstract} We introduce the abstract notion of a \emph{smoothable fine compactified Jacobian} of a nodal curve, and of a family of nodal curves whose general element is smooth. Then we introduce the combinatorial notion of a stability assignment for line bundles and their degenerations. 

We prove that smoothable fine compactified Jacobians are in bijection with these  stability assignments.  

We then turn our attention to \emph{fine compactified universal Jacobians}, that is, fine compactified Jacobians for the moduli space $\overline{\mathcal{M}}_g$ of stable curves (without marked points). We prove that every fine compactified universal Jacobian is isomorphic to the one first constructed by Caporaso, Pandharipande and Simpson in the nineties. In particular, without marked points, there exists no fine compactified universal Jacobian unless $\gcd(d+1-g, 2g-2)=1$. \end{abstract}
\maketitle

\tableofcontents

\section{Introduction}

A classical construction from the XIX century associates with every smooth projective curve $X$ its Jacobian (the moduli space of degree~$0$ line bundles on $X$),  a principally polarized abelian variety of dimension $g$. 
The construction carries on to smooth projective families of curves. One challenging problem arises when $X$ ceases to be smooth. In this case the Jacobian can still be constructed, but in general it fails to be proper. A general problem from the mid XX century was to construct well-behaved compactifications of the Jacobian, whose boundary corresponds to degenerate line bundles of some kind.

Many different constructions have been pursued according to the particular generality required and the initial inputs (see for example \cite{igusa}, \cite{oda79}, \cite{AK},   \cite{caporaso}, \cite{simpson}, \cite{panda}, \cite{esteves}); some of these work in the relative case of families as well.

For simplicity here we restrict ourselves to the case where $X$ is a nodal curve. Also, we will fix the horizon of all possible degenerations of line bundles to \emph{torsion-free coherent sheaves of rank~$1$}. Since we will aim to construct \emph{proper} moduli stacks of stable sheaves, without losing in generality we will additionally assume that all sheaves are \emph{simple}. In this generality, the moduli space of sheaves was constructed as an algebraic space by Altman--Kleiman \cite{AK}. Esteves \cite{esteves} later proved it  satisfies the existence part of the valuative criterion of properness. (This moduli space is not of finite type, hence it is not proper, whenever $X$  is reducible).

Most modular constructions of fine compactified Jacobians use some set of instructions (for example coming from GIT) to single out an open subset of the moduli space of simple sheaves choosing certain \emph{stable} elements, to end up with a proper moduli stack. The construction is often followed by the observation that stability of a sheaf only depends on its multidegree and on its locally free locus in $X$, and then that these discrete data obey a collection of axioms (for example, the number of stable multidegrees of line bundles on a nodal curve $X$ equals the complexity of the dual graph of  $X$). 

In this paper we introduce an abstract  notion of a \emph{fine compactified Jacobian} as a \emph{connected, open} subspace of the moduli space of rank~$1$ torsion-free simple sheaves  of some fixed degree  (not necessarily zero) on $X$, which is furthermore \emph{proper} (see Definition~\ref{def:finecompjac}).\footnote{The adjective ``fine'' classically refers to the  existence of a Poincar\'e sheaf.} It was observed in \cite[Section~3]{paganitommasi} that fine compactified Jacobians can be badly behaved in the sense that they can fail to fit into a family for an infinitesimal smoothing of the curve. (This phenomenon already occurs when $X$ has genus~$1$). Thus, we add a smoothability axiom to the objects that we aim to study. Note that our definition of smoothable fine compactified Jacobian includes the modular fine compactified Jacobians constructed in the literature (e.g. those constructed by Esteves \cite{esteves} and by Oda--Seshadri \cite{oda79} and recently studied in \cite{mv}, \cite{meravi}) .

We then prove our first classification result, stating that smoothable fine compactified Jacobians correspond to a combinatorial datum that we call a \emph{stability assignment} (for smoothable fine compactified Jacobians), which keeps track of the multidegree of the elements of the moduli space and of the locus where they are locally free:

\begin{theorem} \label{mainthm}
Let $X$ be a nodal curve. Taking the associated assignment (see Definition~\ref{assocassign}) induces a bijection 
\[
\Set{\begin{array}{l}
       \textrm{Smoothable fine compactified   } \\
       \textrm{ \ \ \ \ \ \ \ Jacobians of } X
     \end{array}
    } \to \Set{\begin{array}{l}
       \textrm{stability assignments for  } X \textrm{ as}\\ \textrm{introduced in Definition~\ref{finejacstab}}
     \end{array}
    }
\]
whose inverse is defined by taking the moduli space of sheaves that are stable with respect to a given stability assignment (Definition~\ref{defstable}).
\end{theorem}
This follows by applying Corollary~\ref{bijection} to the case where $S$ is a DVR and $\mathcal{X}/S$ is a regular smoothing of $X$. The most difficult part is the proof of properness of the moduli space of stable sheaves with respect to an arbitrary stability assignment (Lemma~\ref{properoverdelta}).

In fact, Corollary~\ref{bijection}  is an extension of Theorem~\ref{mainthm} to the case of fine compactified Jacobians of \emph{families} of nodal curves whose generic element is smooth (Definition~\ref{def:familyfinecompjac}). 
The combinatorial notion of a stability assignment for a \emph{family} is introduced in Definition~\ref{familyfinejacstab}, as the datum of a stability assignment for each fiber, with an additional constraint of compatibility under the degenerations that occur in the family (which induce  morphisms of the corresponding dual graphs).

A natural question is whether our abstract definition of fine compactified Jacobians produces new examples. The most general procedure to construct smoothable fine compactified Jacobians that we are aware of is by means of numerical polarizations. This was introduced by Oda--Seshadri \cite{oda79} for the case of a single nodal curve, then further developed by Kass--Pagani \cite{kp2}, \cite{kp3} for the case of the universal family over the moduli space of pointed stable curves (equivalent objects were constructed by Melo in \cite{melouniversal} following Esteves \cite{esteves}). These definitions and constructions are reviewed in Section~\ref{Sec: OS}.

By Theorem~\ref{mainthm}, it is a completely combinatorial (but  hard) question whether every smoothable fine compactified Jacobian is given by a numerical polarization. (Note that ``smoothable'' here is essential, due to the aforementioned genus~$1$  examples in \cite[Section~3]{paganitommasi}. Those examples are not smoothable, whereas all compactified Jacobians obtained from numerical polarizations are smoothable).  The case where the genus of $X$ equals~$1$ was settled in the affirmative in \cite[Proposition~3.15]{paganitommasi}, and in Example~\ref{genus1} we discuss how to extend this to the case where the first Betti number of the dual graph of $X$ equals $1$. In Example~\ref{ibd} we discuss the numerical polarization that induces integral break divisors (slightly generalizing the analogous result by Christ--Payne--Shen \cite{cps} for the case where $X$ is stable). 
\footnote{While this paper was under peer review, Filippo Viviani constructed in \cite[Example~1.27]{viviani} an example, based on the combinatorics of \cite[Example 6.15]{paganitommasi}, of a smoothable fine compactified Jacobian of a nodal curve of genus $3$ that is not induced by any numerical stability condition.}

We resolve in the positive the similar question for the case of the universal curve over $\overline{\mathcal{M}}_g$. 
Without marked points,  fine compactified universal Jacobians are all given by universal numerical polarizations:

\begin{theorem}
    Let $\overline{\mathcal{J}}_{g} \to \overline{\mathcal{M}}_g$ be a degree~$d$ fine compactified universal Jacobian. Then $\gcd(d-g+1, 2g-2)=1$ and there exists a universal numerical polarization $\Phi$ such that $\overline{\mathcal{J}}_{g}=\overline{\mathcal{J}}_{g}(\Phi)$ (as defined in Section~\ref{Sec: OS}). 
\end{theorem}

This follows from Corollary~\ref{maincoroll}. In particular, as we observe in Remark~\ref{maincoroll2}, without marked points, there are no more fine compactified universal Jacobians than the (essentialy equivalent) ones constructed  in the nineties by Caporaso \cite{caporaso}, Pandharipande \cite{panda} and Simpson \cite{simpson}.  

A similar result does not hold in the presence of marked points: in \cite{paganitommasi} the authors produce examples of fine compactified universal Jacobians for $\overline{\mathcal{M}}_{1,n}$ for all $n \geq 6$ that are not obtained from a universal numerical polarization, hence 
 that do not arise from the methods developed by Kass--Pagani \cite{kp3} or by Esteves and Melo \cite{esteves,melouniversal} (more details in Remark~\ref{n>0}). An explicit combinatorial characterization of the collection $\Sigma^d_{g,n}$ of degree~$d$ fine compactified universal Jacobians for $\overline{\mathcal{M}}_{g,n}$ is available via Corollary~\ref{bijection} applied to the universal family over $\overline{\mathcal{M}}_{g,n}$ (see also Definition~\ref{familyfinejacstab} and Remark~\ref{Rem: univ stab}). It would be interesting to interpret each element of $\Sigma^d_{g,n}$ as a (top-dimensional) chamber in some stability space, as was done in \cite{kp3} for the case of compactified universal Jacobians arising from numerical polarizations. We explore these questions in Section~\ref{sec: final}.

Compactified universal Jacobians have recently played a role in enumerative geometry in the theory of the ($k$-twisted) double ramification cycle, see \cite{bhpss}. As realized in \cite{HKP}, whenever a fine compactified universal Jacobian contains the locus $Z$  of line bundles of multidegree zero, the double ramification cycle can be defined as the pullback of $[Z]$, via some Abel--Jacobi section. This perspective plays an important role in \cite{hmpps}, where an extension to a \emph{logarithmic double ramification cycle} is defined as well.  Because there are different fine compactified universal Jacobians containing $Z$,  the double ramification cycle can be equivalently defined as the pullback of $[Z]$ from different spaces, potentially leading to different formulas, hence to relations in the cohomology of the moduli space of curves. Our classification leads to a complete description  of \emph{all} fine compactified universal Jacobians containing $Z$, whereas previously only those obtained via Kass--Pagani's method \cite{kp3} were considered.

The problem of studying the stability space of complexes of sheaves on a projective variety $X$ has attracted a lot of attention in the last two decades, after Bridgeland's breakthrough \cite{bridgeland} (extended to the case of families in \cite{bayer}). Most of the literature has been devoted to the case where $X$ is nonsingular. It is natural to try to explicitly describe this stability space for $X$ singular, and one place to start is assuming that $X$ is a nodal curve. We expect that the combinatorics developed in Theorem~\ref{mainthm} should be regarded as some kind of skeleton of that stability space.

{\bf Added after peer review:} 
After a first draft of this work had been published on the arXiv repository, Filippo Viviani shared with us his preprint \cite{viviani}. His work uses some of our results. Later, Alex Abreu and an anonymous referee kindly informed us of gaps in the proof of some of our results. 
These gaps have been fixed in this version of the paper, by also employing \cite[Theorem~1.20]{viviani}. The latter does not use any of our results as part of its proof.

\subsection{Acknowledgments}

Many of the initial ideas of this project are owed to Jesse Kass. It is with great pleasure that we acknowledge his intellectual contribution to this work.

We are grateful to Jonathan Barmak for sharing with us \cite{barmak} a solution to the combinatorial problem that appears in the proof of Lemma~\ref{enoughontrees}.

NP is very thankful to Alex Abreu and Filippo Viviani for several important contributions that came after an early version of the preprint was shared with them. In particular, Filippo Viviani pointed out a gap in the proof of the earlier version of Proposition~\ref{cond2}, and Alex Abreu suggested the numerical polarization of Example~\ref{ibd}. NP also wants to thank Marco Fava for several helpful discussions.

We are very grateful to the anonymous referees for the careful reading of the previous version of this paper. Their kind comments  have helped us improve the contents and the exposition.

\section{Notation} 
\label{notation}
Throughout we work with Noetherian schemes over a fixed algebraically closed ground field $k$.

 A  \textbf{curve} over an extension $K$ of $k$ is a $\spec(K)$-scheme $X$ 
 that is proper over $\spec(K)$, geometrically connected, and of pure dimension $1$. The curve $X$ is a \textbf{nodal curve} if it is geometrically reduced and when passing to an algebraic closure $\overline{K}$, its local ring at every singular point is isomorphic to $\overline{K}[[x,y]]/(xy)$. 

A coherent sheaf on a nodal curve $X$ has \textbf{rank~$1$} if its localisation at each  generic point of $X$ has length $1$. It is \textbf{torsion-free} if it has no embedded components.

If $F$ is a rank~$1$ torsion-free sheaf on a nodal curve $X$ we denote by $\NF(F)$ the subset of $X$ where $F$ fails to be locally free. Note that $\NF(F)$ is contained in the singular locus of $X$. If $F$ is a rank~$1$ torsion-free sheaf on $X$ we say that $F$ is \textbf{simple} if its automorphism group is $\mathbf{G}_m$, or equivalently if $X\setminus \NF(F)$ is connected.

A \textbf{family of curves} over a $k$-scheme $S$ is a proper, flat morphism $\mathcal{X} \to S$ whose fibers are curves. A family of curves $\mathcal{X} \to S$ is a \textbf{family of nodal curves} if the fibers over all geometric points are nodal curves.

If $X$ is a nodal  curve over $K$, we  denote by $\Gamma(X)$ its {\bf dual graph} 
i.e. the labelled graph where each vertex $v$ corresponds to an irreducible component $X_{\overline{K}}^v$ of the base change of $X$ to (the spectrum of) an algebraic (equivalently, a separable) closure $\overline{K}$, and edges corresponding to the nodes of $X_{\overline{K}}$.  Note that if $X^v_K$ is an irreducible component defined over $K$, then it is also defined over any extension $L$ of $K$, and the corresponding vertices of the dual graphs  $\Gamma(X_K)$ and $\Gamma(X_L)$ are canonically identified.

The dual graph is labelled by the geometric genus
$p_g(X^v_{\overline{K}})$.  The definition of dual graph extends to the case where $(X,p_1,\dots,p_n)$ is an $n$-pointed  curve. In this case the dual graph $\Gamma(X)$ also has $n$ half-edges labelled from $1$ to $n$, corresponding to the marked points $p_1, \ldots, p_n$.
 We refer to \cite{acg2} and \cite{memoulvi} for a detailed definition and for the notion of graph morphisms. 

 Recall from \cite[\S~7.2]{CCUW} that if $\mathcal{X}/S$ is a family of nodal curves and $s,t$ are geometric points of $S$, then every \'etale specialization of $t$ to $s$ (written as $t \rightsquigarrow s$) induces a morphism of dual graphs $\Gamma(X_s)\rightarrow \Gamma(X_t)$. (For the definition of \'etale specialization in this context we refer to \cite[Appendix~A]{CCUW}).

For a graph $G$ and $H$ a subgraph of $G$, we denote by $G\setminus H$ and by $G/H$ the graph obtained from $G$ by removing the edges of $H$ and the graph obtained from $G$ by contracting the edges of $H$, respectively. 

 We denote by $c(G)$ the {\bf complexity} of the graph $G$, i.e. the number of spanning trees in $G$. (In particular, if $G$ is disconnected, then $c(G)=0$).

If $G$ is a graph and $V$ is a subset of $\Verts(G)$, we denote by $\Gamma(V)$ the induced subgraph on the vertex set $V$.  We denote by $\Edges(V,V^c)$ the subset of $\Edges(G)$ of the edges that connect some element of $V$ to some element of $V^c= \Verts(G) \setminus V$ (equivalently, $\Edges(V,V^c)$ consists of the edges of $G$ that are neither in $\Gamma(V)$ nor in $\Gamma(V^c)$).

We will  denote by $\Delta$ the spectrum of a DVR with residue field $K$, and by $0$ (resp. by $\eta$) its closed (resp. its generic) point. A \textbf{smoothing}  of a nodal curve $X/K$ over $\Delta$ is a flat family $\mathcal{X}/\Delta$  whose generic fiber $\mathcal{X}_{\eta}/\eta$ is smooth and with an isomorphism of $K$-schemes $\mathcal{X}_0 \cong X$.  The smoothing is \textbf{regular} if so is its total space $\mathcal{X}$.

A \textbf{family of rank~$1$ torsion-free sheaves} over a family of curves $\mathcal{X} \to S$ is a coherent sheaf on $\mathcal{X}$, flat over $S$, whose fibers over the geometric points have rank~$1$ and are torsion-free.

If $F$ is a rank~$1$ torsion-free sheaf on a nodal curve $X$ with irreducible components $X_i$, we denote by $F_{\widetilde{X}_{i}}$  the maximal torsion-free quotient of the pullback of $F$ to the normalization $\widetilde{X}_i$ of $X_i$, and then define the {\bf multidegree} of $F$ by \[{\underline{\deg}}(F) := ({\deg}(F_{\widetilde{X}_{i}})) \in \Z^{\operatorname{Vert}(\Gamma(X))}.\]  We define the {\bf (total) degree} of $F$ to be $\deg_X(F):=\chi(F)-1+p_a(X)$, where $p_a(X)= h^1(X, \mathcal{O}_X)$ is the arithmetic genus of $X$. The total degree and the multidegree of $F$ are related by the formula $\deg_X(F) = \sum \deg_{X_i} F + \delta(F)$, where $\delta(F)=\# \NF(F)$ denotes the number of nodes of $X$ where $F$ fails to be locally free.
\footnote{Note that in the Notation section of the papers \cite{kp2,kp3,paganitommasi}  the last equation is incorrectly written with a minus sign: $- \delta(F)$ instead of the correct $+\delta(F)$.}

If $X' \subseteq X$ is a subcurve (by which we will always mean a union of irreducible components), then $\deg_{X'}(F)$ is defined as $\deg (F_{X'})$, where $F_{X'}$ is the 
maximal torsion-free quotient of $F \otimes \mathcal{O}_{X'}$. The total degree on $X$ is related to the degree on a subcurve by the formula
\begin{equation} \label{deg:subcurve}
\deg_X(F)= \deg_{X'}(F) + \deg_{\overline{X \setminus X'}}(F) + \# (\NF(F) \cap (X' \cap \overline{(X \setminus X')})).
\end{equation}
where the overline denotes the (Zariski) closure.

From now on we fix an integer $d$ once and for all.

\subsection{Spaces of multidegrees} Here we define the space of multidegrees on a graph $\Gamma$ at a connected spanning subgraph as the set of all possible ways of labelling its vertices with integral weights, suitably organised by total weight (or total degree). We then define the notion of an assignment on $\Gamma$.

Let $\Gamma$ be a graph and let $\gammazero \subseteq \Gamma$ be a spanning subgraph of $\Gamma$. 
We will denote by $n_\Gamma(\gammazero)$ or simply by $n(\gammazero)$ the number of elements in $\operatorname{Edges}(\Gamma) \setminus \operatorname{Edges}(\gammazero)$. 

 Define the \emph{space of multidegrees} of total degree $d$ of $\Gamma$ at $\gammazero$ as the set
\begin{equation} \label{sgamma}
S^d_{\Gamma}(\gammazero) := \Set{\underline{\mathbf d} \in \Z^{\Verts(\Gamma)} : \sum_{v \in \Verts(\Gamma)} \underline{\mathbf d}(v) = d  - n(\gammazero)} \subset \Z^{\Verts(\Gamma)}.
\end{equation}

The elements of $S^d_{\Gamma}(\Gamma)$ are also known as \emph{degree $d-n(\gammazero)$ divisors} on  $\Gamma$.

According to our convention for the multidegree, if $X$ is a nodal curve with dual graph $\Gamma$ and $F$ is a rank~$1$ torsion-free sheaf on $X$, then the subgraph  $\gammazero(F)$ obtained from $\Gamma$ by removing the edges $\NF(F)$ is a spanning subgraph of $\Gamma$, and we have
\[\underline{\deg}(F) \in S^d_{\Gamma}(\gammazero(F)). \]

 We are now ready for the definition of an assignment on a graph $\Gamma$, which will play a role in the definition of a stability assignment (Definition~\ref{finejacstab}). 
\begin{definition} \label{def: assignment}
   {\bf A degree~$d$ assignment for the graph $\Gamma$} is a subset
\[\sigma = \{(\gammazero,\underline{\mathbf d}):\;\underline{\mathbf d}\in S^d_\Gamma(\gammazero)\}\subset \{\text{connected spanning subgraphs of }\Gamma\}\times \Z^{\Verts(\Gamma)}.\]

If $\sigma$ is a degree~$d$ assignment for the graph $\Gamma$ and $\gammazero \subseteq \Gamma$ is a subgraph, we define
\[\sigma(\gammazero) := \{\underline{\mathbf d}:\;(\gammazero,\underline{\mathbf d})\in\sigma\} \subset S^d_{\Gamma}(\gammazero)\] 
(and the latter is empty unless $\gammazero$ is connected and spanning).
\end{definition}

\section{Fine compactified Jacobians}

\label{secfinejac} 
In this section we introduce the notion of a (smoothable) fine compactified Jacobian for a nodal curve (Definition~\ref{def:finecompjac}), and for a flat family of nodal curves (Definition~\ref{def:familyfinecompjac}).

Let $\mathcal{X}/S$ be a flat family of nodal curves over a $k$-scheme $S$. Then there is an algebraic space $\Pic^d(\mathcal{X}/S)$ parameterizing line bundles on $\mathcal{X}/S$ of relative degree~$d$ (see \cite[Chapter~8.3]{blr}).   By \cite{AK} and \cite{esteves} the space $\Pic^d(\mathcal{X}/S)$ embeds in an algebraic space $\Simp^{d}(\mathcal{X}/S)$ parameterizing  flat families of degree~$d$ rank~$1$ torsion-free simple sheaves on $\mathcal{X}/S$. The latter is locally of finite type  over $S$ and satisfies the existence part of the valuative criterion of properness.  However, it can fail to be of finite type and separated. 
In the special case of $S=\Spec(K)$ and $X=\mathcal{X}/S$ we will simply write $ \Pic^d(X)$ (resp. $\Simp^d(X)$) for $\Pic^d(\mathcal{X}/S)$ (resp. for $\Simp^d(\mathcal{X}/S)$).

Let $X$ be a nodal curve over some field extension $K$ of $k$. The main point of this paper is to describe well-behaved subspaces of $\Simp^{d}(X)$,
generalizing existing notions of compactified Jacobians in the literature.  Specifically, we study the following subschemes.

\begin{definition} \label{def:finecompjac}
	A \textbf{degree $d$ fine compactified Jacobian} 
 is a geometrically connected open subscheme $\fcj\subseteq\Simp^{d}(X)$ that is proper over $\spec(K)$.

We say that the fine compactified Jacobian $\fcj$ is {\bf smoothable} if there exists a regular smoothing $\mathcal X\to\Delta$ of $X$, where $\Delta$ is the spectrum of a DVR  with residue field $K$, such that $\fcj$ is the fiber over $0\in\Delta$ of an open and $\Delta$-proper subscheme of $\Simp^{d}(\mathcal{X}/\Delta)$.

\end{definition}
Note that the fiber over the generic point $\eta$ of a nonempty, open and $\Delta$-proper subscheme of $\Simp^{d}(\mathcal{X}/\Delta)$ is necessarily the moduli space of degree $d$ line bundles $\Pic^d(\mathcal{X}_{\eta}/{\eta})$. As openness and properness are stable under base change, the fiber over $0 \in \Delta$ 
 is open in $\Simp^d(X)$ and $K$-proper. The axiom  ``geometrically connected'' is redundant in the smoothable case, because the moduli space of degree~$d$ line bundles on the generic point is geometrically connected and dense in $\Simp^{d}(\mathcal{X}/\Delta)$. 

\begin{remark} It follows from Lemma~\ref{properoverdelta} (combined with Proposition~\ref{exists-forall}) that requiring  the subscheme $\fcj \subseteq \Simp^d(X)$ to extend to \emph{some regular} smoothing of the curve is equivalent to requiring that it extends to \emph{all} smoothings.
\end{remark}

In this paper, we will focus on smoothable fine compactified Jacobians, since they are better behaved and occur more often in applications.
\begin{remark}
When $K=k$ is algebraically closed and $\operatorname{char}(k)=0$, Definition~\ref{def:finecompjac} coincides with \cite[Definitions~2.1, 2.4]{paganitommasi}  (by passing to the completion of the DVR).

    In \cite[Section~3]{paganitommasi} the authors give a complete classification of fine compactified Jacobians of curves of genus~$1$, showing in particular the existence of nonsmoothable examples.
\end{remark}

\begin{example} \label{ex: irred}
If $X$ is a geometrically irreducible curve over $K$, then $\Simp^d(X)$ is proper over $K$, so the only degree~$d$ fine compactified Jacobian is $\Simp^d(X)$ itself. These Jacobians are always smoothable. (See Examples~ \ref{ex: stab-irred} and \ref{irred-is-OS} for the corresponding unique stability assignment).
\end{example}

\begin{example} \label{ex: vinecurves}
In the case of curves with two geometrically irreducible components, fine compactified Jacobians are no longer irreducible. Assume for simplicity that $X$ is a {\bf vine curve of type $t$}, that is, the union of two nonsingular curves intersecting transversely at $t$ nodes.
 We will later see in Example~\ref{vinecurves-nonsm} that every fine compactified Jacobian of $X$ is smoothable, and that it consists of $t$ irreducible components whose generic points correspond to line bundles of consecutive bidegrees.
\end{example}

\begin{remark}
    \label{sec:strat} The moduli space $\Simp^d(X)$ of a nodal curve $X$ admits a natural stratification (see for example \cite{mv}) into locally closed subsets
    \begin{equation} \label{eq: strata} \Simp^d(X) = \bigsqcup_{(\Gamma_{0},\underline{\mathbf d})}\mathcal J_{(\Gamma_{0},\underline{\mathbf d})}\end{equation}
where the union runs over  all connected spanning subgraphs $\Gamma_{0}\subseteq \Gamma(X)$ and all multidegrees $\underline{\mathbf d}\in S^d_{\Gamma(X)}(\gammazero)$.
Each subspace $\mathcal J_{(\Gamma_{0},\underline{\mathbf d})}\subset \Simp^d(X)$ is defined as the locus whose points are sheaves $F$ that fail to be locally free on $\NF(F)=\Edges(\Gamma) \setminus \Edges(\Gamma_{0})$, and whose multidegree $\underline{\deg}(F)$ equals $ \underline{\bf d}$.
\end{remark}

We now extend the notion of a fine compactified Jacobian to the case of a family of nodal curves $\mathcal{X}/S$. Recall that the moduli space $\Simp^d(\mathcal{X}/S)$ is also defined in \cite{melouniversal} when $\mathcal{X}/S$ is the universal curve $\Cb_{g,n}/\Mb_{g,n}$ over the moduli stack of stable $n$-pointed curves of arithmetic genus $g$. In this case $\Simp^d(\Cb_{g,n}/\Mb_{g,n})$ is a Deligne--Mumford stack representable (by algebraic spaces) and flat over $\Mb_{g,n}$.

\begin{definition} \label{def:familyfinecompjac} Assume that $S$ is irreducible with generic point $\theta$, and assume that the generic fiber $\mathcal{X}_\theta/\theta$ is smooth.

A \textbf{family of degree~$d$ fine compactified Jacobians} for the family $\mathcal{X}/S$ is an open algebraic subspace  $\fcjs \subseteq \Simp^d(\mathcal{X}/S)$ that is proper over $S$.

We say that a degree~$d$ fine compactified Jacobian $\fcuj \subset \Simp^d(\Cb_{g,n}/\Mb_{g,n})$ is a \textbf{degree~$d$ fine compactified universal Jacobian for the universal curve over $\Mb_{g,n}$}. (We will often omit to specify ``for the universal curve over $\Mb_{g,n}$'', when clear from the context).
\end{definition}

Note that the assumption that the generic fiber is smooth implies that all fibers over $S$ of a degree~$d$ fine compactified Jacobian are smoothable.

\section{Stability assignments for smoothable fine compactified Jacobians} 

Here we define the combinatorial data identifying \emph{smoothable} fine compactified Jacobians. We first do so for a single nodal curve, that is, for a fixed dual graph (Definition~\ref{finejacstab}), and then we generalize the definition to families (Definition~\ref{familyfinejacstab}). 

If $X$ is a nodal curve over an algebraically closed field, a sheaf $F\in\Simp^d(X)$ has two natural combinatorial invariants, given by its multidegree and by the subset $\NF(F)\subseteq \operatorname{Sing}(X)=\Edges(\Gamma(X))$ of points of the curve where $F$ fails to be locally free.  Hence it makes sense to study a fine compactified Jacobian $\fcj$ on $X$ by looking at all pairs $\left(\gammazero(F)= \Gamma(X) \setminus \NF(F),\underline{\deg}(F)\right)$ with $F\in\fcj$. Recall that, with the notation introduced in Equation~\ref{sgamma}, we can regard $\underline{\deg}(F)$ as an element of  the space of multidegrees $S_{\Gamma(X)}^d(\gammazero(F))$.


For a single curve $X$, we identify, in Definition~\ref{finejacstab}, the two properties characterizing the set of such pairs. One is related to properness, combined with the smoothability of the Jacobian, and it requires that the set of stable multidegrees  should be a minimal complete set of representatives for the natural chip-firing action on the dual graph (see Definition~\ref{def:chipfiring}). The other corresponds to openness. In combinatorial terms, this means that if we add an edge $e$ to $\gammazero$, the set of stable multidegrees on $\Gamma\cup\{e\}$ should contain all multidegrees obtained by ``adding a chip'' to either endpoints of $e$.

For a family of curves, we further require compatibility with all contractions of the dual graphs involved. 

\subsection{Stability assignments for a single curve}

We start by introducing the twister group of a graph, which will play a role in characterizing smoothable compactified Jacobians.
\begin{definition}\label{def:chipfiring}
Let $G$ be a graph.
For each $v \in \Verts(G)$, define the \emph{twister of} $G$ at $v$ to be the element of $\Z^{\Verts(G)}$ defined by
\[
\tw_{G, v} (w) = \begin{cases}  \text{ \# of edges of } G \text{ having } v \text{ and } w \text{ as endpoints }& \text{when } w \neq v, \\ - \text{ \# of nonloop edges of } G \text{ having } v \text{ as an endpoint} & \text{when } w =v.\end{cases}
\]

The \textbf{twister group} (or chip-firing group) $\tw(G)$ is the subgroup of $\Z^{\Verts(G)}$ generated by the set $\Set{\tw_{G, v}}_{ v \in \Verts(G)}$. 
\end{definition}

Recall from Equation~\eqref{sgamma} the definition of the space of multidegrees $S^d_{G}(G)$ of total degree equal to $d$. The twister group of $G$ is contained in the sum zero submodule $S^0_{G}(G)$ of $\Z^{\Verts(G)}$. Hence the group structure on $\Z^{\Verts(G)}$ restricts to an action of $\tw(G)$ on $S^d_{G}(G)$. The quotient  group
\begin{equation}\label{jacobiangraph} J^d(G):=S^d_{G}(G)/\tw(G)\end{equation}
is then a torsor over $J^0(G)$, which is a finite abelian group. The latter is also known as the {\bf Jacobian}  of the graph~$G$, and it has a number of element equal to the complexity $c(G)$ of the graph $G$. (It is also known by other names in the literature, 
such as the \emph{degree class group}, or the \emph{sandpile group}, or the \emph{critical group} of the graph $G$).

\begin{remark} \label{twister} Let $\mathcal{X}$ be a regular smoothing of $X$ over some discrete valuation ring $\Delta$ with generic point $\eta$. Let $T$ be the image under the restriction map  $\pic(\mathcal X)\rightarrow \pic(X)$ of the kernel of the surjection $\pic(\mathcal X)\rightarrow \pic(\mathcal X_\eta)$ (the restriction to the generic point). We claim that the restriction to $T$ of the multidegree homomorphism $\pic(X) \to \Z^{\Verts(\Gamma)}$  defines an isomorphism $T \to \tw(\Gamma)$.  


Since $\mathcal X$ is regular, 
the irreducible components $\{X_v\}_{v\in\Verts(\Gamma)}$ of $X$ are Cartier divisors on $\mathcal X$ and it is easy to check that the elements of $T$ are of the form 
\[
\mathcal O_{\mathcal X}\left(\sum_{v\in\Verts(\Gamma)}{d_vX_v}\right)\otimes \mathcal O_X
\]
with $\underline{\mathbf d}\in\Z^{\Verts(\Gamma)}$. One can explicitly compute that the restriction to $T$ of the multidegree map $\pic(X)\rightarrow\Z^{\Verts(\Gamma)}$ is given by
\[\begin{array}{rcl}
\mathcal O_{\mathcal X}\left(\sum_{v\in\Verts(\Gamma)}{d_vX_v}\right)\otimes \mathcal O_X&\longmapsto&\sum_{v\in\Verts(\Gamma)}d_v\tw_{\Gamma,v}.
\end{array}\]
This homomorphism  is injective by \cite[(5.2)]{raynaud70}, and its image is $\tw(\Gamma)$ (by definition). 
This proves that $T \to \tw(\Gamma)$ is an isomorphism. In particular, $T$ as an abstract group is independent of the choice of the regular smoothing (though one could see that the embedding $T \subset \Pic(X)$ is not independent of the smoothing).
 \end{remark}

We are now ready to define our notion of a stability assignment.
\begin{definition} \label{finejacstab} 

A {\bf degree~$d$ smoothable fine compactified Jacobian stability assignment}, or shortly a {\bf degree $d$ stability assignment}  for the graph $\Gamma$ is a degree~$d$ assignment for $\Gamma$ (as in Definition~\ref{def: assignment}) that satisfies the following conditions:
\begin{enumerate}
\item\label{firstfjs} For all edges $e$ of $\Edges(\Gamma)\setminus\Edges(\gammazero)$ with endpoints $v_1$ and $v_2$ we have
\[
(\gammazero,\underline{\mathbf d})\in\sigma \Rightarrow (\gammazero\cup\{e\},\underline{\mathbf d}+\underline{\mathbf e}_{v_1}),(\gammazero\cup\{e\},\underline{\mathbf d}+\underline{\mathbf e}_{v_2})\in\sigma, 
\]
where $\underline{\mathbf e}_{v_i}$ denotes the vector in the standard basis of $\Z^{\Verts(\Gamma)}$ corresponding to $v_i$.
\item\label{secondfjs} For every connected spanning subgraph $\gammazero$, the subset  
\[\sigma(\gammazero) := \{\underline{\mathbf d}:\;(\gammazero,\underline{\mathbf d})\in\sigma\} \subset S^d_{\Gamma}(\gammazero)\] is a minimal complete set of representatives for the action of the twister group $\tw(\gammazero)$ on $S^d_{\Gamma}(\gammazero)$.
\end{enumerate}

If $X$ is a nodal curve, a degree $d$ stability assignment on $X$ is  a degree $d$ stability assignment on its dual graph $\Gamma(X)$.
\end{definition}
Note that the genus of each of the components of $X$ does not play any role in the above definition.

\begin{remark} \label{complexity} 
With the notation as above, the number of elements of  $\sigma(\gammazero)$ equals the number of elements of the Jacobian $J^0(\gammazero)$. 
By the Kirchoff--Trent theorem, this number equals the complexity $c(\gammazero)$ of the graph $\gammazero$. In particular, it is finite.
\end{remark}

It is in general hard to classify all stability assignments on a given stable graph. However, the task is within reach when the number of vertices is small.
\begin{example} \label{ex: stab-irred}
If $\Gamma$  only has $1$ vertex (i.e. it is the dual graph of an irreducible curve), there is exactly $1$ stability assignment on $\Gamma$. If there are $t$ edges, the unique degree~$d$ stability assignment is \[\sigma=\bigcup_{0 \leq i \leq t} \bigcup_{E \subseteq \Edges(\Gamma), |E|=i} \set{(\Gamma \setminus E,d-i)}.\]
\end{example}
\begin{example}
\label{vinestab}
If instead $\Gamma$ consists of $2$ vertices $v_1, v_2$ connected by $t$ edges, and no other edges (i.e. $\Gamma$ is the dual graph of a vine curve of type $t$, see Example~\ref{ex: vinecurves}), let  $\Gamma_1, \ldots, \Gamma_t$ be the spanning trees. Then it follows from the definition that for every stability assignment $\sigma$ there exists a unique integer $\lambda$ such that

\begin{enumerate}
    \item  $\sigma(\Gamma_i)=\set{(\lambda,d+1-\lambda-t)}$ for all $i=1, \ldots, t$;
\item $\sigma(\Gamma)=\set{(\lambda, d-\lambda) ,(\lambda+1, d-\lambda-1), \ldots, (\lambda+t-1, d+1-\lambda-t)}$. \end{enumerate}\end{example}

The following result will be used later as a tool to prove that certain collections are stability assignments. It was first proved by Barmak~\cite{barmak}. 
A first publicly available proof of part of this result, based upon Barmak's argument, appeared in Yuen's Phd thesis,  \cite[Theorem~3.5.1]{chiho}.

\begin{proposition} \label{barmak} (\cite[Theorem~1.17 Part~(1) and (2.a)]{viviani}).
    Let $\sigma$ be a degree~$d$ assignment on the graph $\Gamma$ (see Definition~\ref{def: assignment}), and assume that $\sigma$ satisfies Part~(1) of Definition~\ref{finejacstab}, and that $\sigma(T)$ is nonempty for every spanning tree $T \subseteq \Gamma$. 
    
    Then for all spanning subgraphs $\gammazero \subseteq \Gamma$, we have $|\sigma(\gammazero)| \geq c(\gammazero)$ (the complexity of $\gammazero$). Moreover, if $|\sigma(\Gamma)|=c(\Gamma)$, then $|\sigma(\gammazero)|= c(\gammazero)$ for each spanning subgraph $\gammazero \subseteq \Gamma$.
\end{proposition}

\subsection{Stability assignments for families} \label{Sec: families} So far we have  discussed the notion of a stability assignment for a curve in isolation. The flatness condition for families of sheaves imposes an additional compatibility constraint.

\begin{definition} \label{combo-family} Let $f\colon \Gamma \to \Gamma'$ be a morphism of graphs and let $\sigma$ and $\sigma'$ be degree $d$ stability assignments on $\Gamma$ and $\Gamma'$, respectively. 
We say that $\sigma$ is $f$-compatible with  $\sigma'$ if for every connected spanning subgraph $\gammazero$ of $\Gamma$, we have
\[(\gammazero,\underline{\mathbf d})\in\sigma \implies (\gammazero',\underline{\mathbf d'}) \in \sigma',  \]
where $\gammazero'$ is the image of $\gammazero$ under $f$ (and therefore it is a connected spanning subgraph of $\Gamma'$), and $\underline{\mathbf d'}$ is defined by
\begin{equation} \label{comp-contractions}
\underline{\mathbf d'}(w) = \sum_{f(v)=w} \underline{\mathbf d}(v) + \#  \left\{\textrm{edges of } \Gamma\setminus\gammazero \textrm{ that are contracted to } w \textrm{ by } f\right\}.
\end{equation}
\end{definition}

Note that the notion of $f$-compatibility only depends upon the map that $f$ induces on the set of vertices of the two graphs. We are now ready for the definition of the compatibility constraint.

\begin{definition} \label{familyfinejacstab}
Let $\mathcal{X}/S$ be a family of nodal curves. A \textbf{family of degree $d$ stability assignments (for fine compactified Jacobians)} for $\mathcal X/S$ 
consists of associating a degree $d$ stability assignment $\sigma_s$ on $\Gamma(X_s)$ (as in Definition~\ref{finejacstab}) with every geometric point $s$ in $S$, in a way that is compatible with all morphisms $\Gamma(X_s) \to \Gamma(X_t)$ arising from any \'etale specialization $t \rightsquigarrow s$ occurring on $S$.
\end{definition}

We will say that a family of degree~$d$ stability assignments for the  universal family $\overline{\mathcal{C}}_{g,n}$ over $\Mb_{g,n}$ is a {\bf degree~$d$ universal stability assignment of type $(g,n)$} (and often omit ``of type $(g,n)$'' when clear from the context). These stability assignments will be studied in Section~\ref{sec: univ}. 

Universal stability assignments can be defined purely in terms of the category $G_{g,n}$ of stable graphs:

\begin{remark} \label{Rem: univ stab}
In the case of universal stability assignments, the \'etale specializations of points of $\Mb_{g,n}$ induce all morphisms in the category $G_{g,n}$ of stable graphs of genus $g$ with $n$ marked half-edges. 

In particular, if $\sigma$ is a degree~$d$ universal stability assignment of type $(g,n)$ and  $\alpha \colon \Gamma(X_1) \to \Gamma(X_2)$ is an isomorphism of the dual graphs of two pointed curves $[X_1],[X_2]\in\Mb_{g,n}$, then $\alpha$ identifies $\sigma_{[X_1]}$ with $\sigma_{[X_2]}$.

We conclude that a degree $d$ universal stability assignment of type $(g,n)$  is a collection $\Set{\sigma_\Gamma}_{\Gamma \in G_{g,n}}$ such that $\sigma_\Gamma$ is a degree $d$ stability assignment on $\Gamma$ and $\sigma_{\Gamma}$ is $f$-compatible with $\sigma_{\Gamma'}$ for any morphism $f:\;\Gamma\rightarrow \Gamma'$ in $G_{g,n}$.
\end{remark}

\section{Combinatorial preparation: perturbations and lifts}

This section contains combinatorial results that will be used in later proofs. First we introduce the notion of a perturbation for the multidegree of a sheaf that fails to be locally free at some nodes. This corresponds to deforming a simple sheaf into a line bundle (Lemma~\ref{impliedbyopen}). Then we investigate the relationship between the chip-firing action on a graph $\Gamma$ and the chip-firing action on the graph obtained from $\Gamma$ by subdividing each of its edges a certain number of times. In terms of dual graphs, this construction corresponds to blowing up a certain number of times the nodes of our curve. Our main techincal result here is Proposition~\ref{coroll: chipfiring}.

\begin{definition} \label{def: pert}
Let $\gammazero \subseteq \Gamma$ be a connected spanning subgraph. 
A \emph{$\gammazero$-perturbation} is  an element   in $S^{n(\gammazero)}_{\gammazero}(\gammazero)$ (for $n(\gammazero)=|\operatorname{Edges}(\Gamma) \setminus \operatorname{Edges}(\gammazero)|$) of the form \[ \sum_{e \in \Edges(\Gamma) \setminus \Edges(\gammazero)} {\underline{\mathbf e}}_{t(e)} \] where $t$ is some choice of an orientation on $\Edges(\Gamma) \setminus \Edges(\gammazero)$. 
\end{definition}
(If $D$ is a $G$-perturbation for some $G$, then $D$ is also known in the literature as a semibreak divisor, see \cite[Definition~3.1]{gst}).

By the description in Remark~\ref{sec:strat} of the stratification of $\Simp^d(X)$, and by Lemma~\ref{impliedbyopen},  if $\gammazero$ is a connected spanning subgraph of $\Gamma$, a $\gammazero$-perturbation is the same as the multidegree of a line bundle on $X$ that specializes to a sheaf $F$ with $\gammazero(F)=\gammazero$, $\underline{\deg}(F)=\underline{\mathbf 0}$. We are now ready for the following technical lemma.

\begin{lemma} \label{impliedbyopen}
Let $\mathcal X$ be a family of nodal curves over the spectrum $\Delta$ of a DVR with algebraically closed residue field $\overline{K}$, and consider a section $\sigma:\;\Delta\rightarrow \Simp^d(\mathcal X/\Delta)$. 
Let us fix a geometric point $\overline\eta$ lying over the generic point of $\Delta$ and denote by $\Gamma$ and $\Gamma'$ the dual graphs of the geometric fibers $X_0$ and $X_{\overline\eta}$ of $\mathcal X$, respectively, and by $f:\;\Gamma\rightarrow \Gamma'$ the morphism of graphs induced by the specialization of $\overline\eta$ to $0$. 
Write $F_0$ and $F_{\overline{\eta}}$ for~$\sigma(0)$  and $\sigma(\overline\eta)$, respectively, and set $\gammazero=\gammazero(F_0)\subseteq \Gamma$ and $\gammazero'=\gammazero(F_{\overline{\eta}})\subseteq \Gamma'$ for the connected subgraphs where the corresponding sheaf is locally free.  Then  there exists a $\gammazero\subseteq f^{-1}(\gammazero')$-perturbation $T$ such that the multidegrees $\underline{\mathbf d}=\underline{\deg}(F_0)$ and $\underline{\mathbf d}' = \underline{\deg}(F_{\overline\eta})$ satisfy the relation \[\underline{\mathbf d}'= f\left(\underline{\mathbf d}+T\right).\]
\end{lemma}

Note that if $\gammazero'\subseteq \Gamma'$ is connected and spanning, then so is $f^{-1}(\gammazero') \subseteq \Gamma$. 

\begin{proof}
The section $\sigma$ defines a rank $1$ sheaf $\mathcal F$ over $\mathcal X/\Delta$. By \cite[Proposition~5.5]{esteves-pacini}, the projectivization $\mathcal Y=\mathbf P_{\mathcal X}(\mathcal F)$ is a family of curves over $\Delta$ and there exists a line bundle $\mathcal L$ over $\mathcal Y/\Delta$ such that $\mathcal F = \psi_*\mathcal L$ holds for the natural map $\psi:\;\mathcal Y\rightarrow \mathcal X$. Geometrically, the restriction of $\psi$ to $Y_{\overline\eta}\rightarrow X_{\overline\eta}$ is the blow up that corresponds to adding a genus $0$ vertex on all edges of $\Gamma'$ that do not belong to $\gammazero'$. The description of $\psi_0:\;Y_0\rightarrow X_0$ in terms of $\Gamma$ and $\gammazero$ is completely analogous. Moreover, the multidegrees of $L_{\overline\eta}$ and $L_{0}$ can be obtained from the multidegrees of $F_{\overline\eta}$ and $F_{0}$, respectively, by taking $\underline{\mathbf d}'$ (resp. $\underline{\mathbf d}$) and assigning degree $1$ to the additional vertices. Then the  claim follows from the fact that $L_0$ is a specialization of the line bundle $L_{\overline\eta}$.
\end{proof}

We will now prove some combinatorial ingredients that relate the chip-firing on a graph  with the chip-firing on its  blow up.  We fix a graph $\Gamma$ and a function $m \colon \Edges(\Gamma) \to \mathbf{N}$ and let ${\tGamma}_m$ be the graph obtained  by subdividing each edge $e$ of $\Gamma$ into $m(e)+1$ edges (called an exceptional chain) by adding $m(e)$  vertices (called exceptional vertices) in the middle.

We start by relating the complexities of $\Gamma$ and $\tGamma$.


\begin{lemma} (\cite[Theorem~3.4]{bms})  \label{complexityblowup}
 The following formula relates the number of spanning trees (i.e. the complexity) of ${\tGamma}_m$ with the complexity of the spanning subgraphs of $\Gamma$:

\begin{equation} \label{eq: number}
c({\tGamma}_m)= \sum_{\substack{G \subseteq \Gamma  \text{ connected}\\ \text{ spanning subgraph}}} \left( \prod_{e \in \Edges(\Gamma) \setminus \Edges(G)} m(e)\right) \cdot c(G).
\end{equation}
\end{lemma}

In order to  state and then prove Proposition~\ref{coroll: chipfiring},  we first need to define the lift of an assignment to $\widetilde{\Gamma}$. 

Let $\sigma$ be a degree~$d$ assignment for $\Gamma$ (see Definition~\ref{def: assignment}).
For each $m \colon \Edges(\Gamma) \to \mathbf{N}$, let $\widetilde{\sigma}_m$, we then define the lift $\widetilde{\sigma}_m \subset S^d_{\widetilde{\Gamma}}(\widetilde{\Gamma})$
of $\sigma$ as follows. First of all, for every element $(\Gamma,\underline{\mathbf d}) \in \sigma$ we can naturally identify $\underline{\mathbf d}$ with the  element of $\Z^{\Verts(\widetilde{\Gamma})}$ obtained by extending $\underline{\mathbf{d}}$ to zero on all exceptional vertices. Then a lift of $(\gammazero,\underline{\mathbf d})$ is any element $\widetilde{\underline{\mathbf d}}\in S^d_{\widetilde{\Gamma}}(\widetilde{\Gamma})$ of the form
\begin{equation}\label{lift}\widetilde{\underline{\mathbf d}} = \underline{\mathbf d}+\sum_{e\in \Edges(\Gamma)\setminus\Edges(\gammazero)}\underline{\mathbf e}_{v(e)},
\end{equation}
for any choice of a function $v\colon \Edges(\Gamma)\setminus\Edges(\gammazero) \to \Verts(\widetilde{\Gamma})$ such that each vertex $v(e)$ is one of the $m(e)$ interior vertices of the exceptional chain 
in $\widetilde{\Gamma}$ that corresponds to the edge $e$.

Then we have:

\begin{proposition} \label{coroll: chipfiring}
Let $\sigma$ be a degree~$d$ assignment on $\Gamma$ that satisfies Part~(1) of Definition~\ref{finejacstab}. Then:

\begin{enumerate} \item \label{uno} Let $m \colon \Edges(\Gamma) \to \mathbf{N}$ be the constant function equal to $1$, and let $\widetilde{\Gamma}=\widetilde{\Gamma}_m$ and $\tilde{\sigma}=\tilde{\sigma}_m$ be as defined above. Assume that  $|\sigma(G)| \geq c(G)$ for all connected spanning subgraphs $G$ of $\Gamma$. If $\widetilde{\sigma}\subset S^d({\tGamma})$  is a minimal complete set of representatives for the action of $\tw({\tGamma})$, then the assignment $\sigma$ is a degree~$d$ stability assignment for $\Gamma$.

\item \label{due} If the assignment $\sigma$ also satisfies Part~(2) of Definition~\ref{finejacstab}, then for all $m \colon \Edges(\Gamma) \to \mathbf{N}$, the lift $\widetilde{\sigma}_m\subset S^d({\tGamma}_m)$ (defined above) is a minimal complete set of representatives for the action of $\tw({\tGamma}_m)$.
\end{enumerate}
\end{proposition}

\begin{proof}

  We first prove  \eqref{uno}, \emph{i.e.} we assume $m$ is constant and equal to $1$ and prove that $\sigma$ also satisfies Part (2) of Definition~\ref{finejacstab}. By \cite[Theorem~1.20, (c) $\implies$ (b)]{viviani} it suffices to prove that $|\sigma(G)|=c(G)$ for all $G \subseteq \Gamma$. By Proposition~\ref{barmak} it suffices to prove the inequality $|\sigma(\Gamma)|\leq c(\Gamma)$. Arguing by contradiction, it is enough to show that if any two different elements ${\bf \underline{e}}$ and ${\bf \underline{e}}'$ of $\sigma(\Gamma)$ are in the same orbit for the action of $\tw(\Gamma)$, then the corresponding lifts ${\bf \underline{\tilde{e}}}$ and ${\bf \underline{\tilde{e}}}'$ to $\widetilde{\sigma}$ via Equation~\ref{lift}  are different and in the same $\tw(\widetilde{\Gamma})$-orbit.
   
   Therefore, we assume that there exists $f \colon \Verts(G) \to \Z$ such that
   \[
{\bf \underline{e}} - {\bf \underline{e}}' = \sum_{v \in \Verts(G)} f(v) \tw_{G,v}.
   \]
     We now define\footnote{This argument is a particular case of a general procedure discussed in \cite[Section~2.4]{aminiesteves}.}  a lift $g \colon \Verts({\tGamma}) \to \Z$ of $f$ as follows:
     \[
     g(v)= \begin{cases} 2 f(v') &  \textrm{ when } v \textrm{ is not exceptional and corresponds to }v' \in \Verts(\Gamma),  \\ f(w)+f(w') &\textrm{ when } v \textrm{ is exceptional and lies between  } w,w' \in \Verts(\Gamma).\end{cases}
     \]
     
     With this definition we have
\[
{\bf \tilde{\underline{e}}}- {\bf \tilde{\underline{e}}}' = \sum_{v \in \Verts({\tGamma})} g(v) \tw_{{\tGamma},v}.
\]
This proves that the lifts ${\bf \tilde{\underline{e}}}$ and $ {\bf \tilde{\underline{e}}}'$, which are different elements of $\widetilde{\sigma}$, are in the same $\tw({\tGamma})$-orbit. This completes our proof of \eqref{uno}.

To prove  \eqref{due}, we begin by   observing that if $\sigma$ is a stability assignment (as in Definition~\ref{finejacstab}), then the number of elements of $\widetilde{\sigma}_m$ equals the right hand side of \eqref{eq: number}, which equals the complexity of ${\tGamma}_m$  by Lemma~\eqref{complexityblowup}. Thus it suffices to prove that any two different elements ${\bf \underline{\tilde{e}}}$ and ${\bf \underline{\tilde{e}}}'$ of $\widetilde{\sigma}_m$   belong to different $\tw(\widetilde{\Gamma}_m)$-orbits. By Lemma~\ref{blah} below, it suffices to prove the inequality
\begin{equation} \label{ineqV}
\left| \sum_{v \in \widetilde{V}} {\bf \underline{\tilde{e}}}(v)-{\bf \underline{\tilde{e}}}'(v) \right|< \left|\Edges(\widetilde{V},\widetilde{V}^c)\right|
\end{equation}
for all $\widetilde{V} \subset \Verts(\widetilde{\Gamma}_m)$ such that the induced subgraphs $\Gamma(\widetilde{V})$ and $\Gamma(\widetilde{V}^c)$ are connected.

Let $j \colon \Verts(\Gamma) \to \Verts(\widetilde{\Gamma})$ be the inclusion of the nonexceptional vertices. If $\widetilde{V}$ is not in the image of $j$, then $\widetilde{V}$ is a subset of an exceptional chain, so the right-hand side of \eqref{ineqV} equals $2$ and by the definition of a lift we have that the left hand side of \eqref{ineqV} equals $0$ or $1$. From now on we will assume that $\widetilde{V}$ contains some nonexceptional vertices.

We now prove \eqref{ineqV} under the additional assumption that ${\bf \underline{\tilde{e}}}$ and ${\bf \underline{\tilde{e}}}'$ are zero on all exceptional vertices, namely that they are lifts of elements  ${\bf \underline{{e}}}$ and ${\bf \underline{{e}}}'$ that belong to $\sigma(\Gamma)$.  Because $\sigma$ is a degree~$d$ stability assignment for $\Gamma$, by \cite[Theorem~1.20, (b) $\implies$ (a)]{viviani}, it is also a ``V-stability condition'' in the sense of \cite{viviani}. By \cite[Definition~1.4]{viviani}), this means that, for all $V \subseteq \Verts(\Gamma)$ such that $\Gamma(V)$ and $\Gamma(V^c)$ are connected, there exist  integers $n_{{\bf \vec{e}}}(V)$ and $n_{{\bf \vec{e}'}}(V)$ such that the inequalities 
\[
n_{{\bf \vec{e}}}(V) \leq \sum_{v \in V} {\bf \underline{{e}}}(v) \leq n_{{\bf \vec{e}}}(V)+ \left|\Edges(V,V^c)\right| -1 
\]
and
\[
n_{{\bf \vec{e}'}}(V) \leq \sum_{v \in V} {\bf \underline{{e}}}'(v) \leq n_{{\bf \vec{e}'}}(V)+ \left|\Edges(V,V^c)\right| -1 
\]
hold. The combination of these two inequalities applied to $V:=j^{-1}(\widetilde{V})$ implies Inequality~\eqref{ineqV}.

    If ${\bf \underline{\tilde{e}}}$ and ${\bf \underline{\tilde{e}}}'$ are not zero on all of the exceptional vertices of $\widetilde{\Gamma}$, we first replace them by elements of $\widetilde{\sigma}_m$ that satisfy this additional hypothesis. This can be achieved by replacing each $1$ on each exceptional vertex in $\tilde{V}$ (resp. in $\tilde{V}^c$) with a $0$, and for each such replacement by adding $+1$ on (one of) the closest nonexceptional vertex in $\Gamma(\tilde{V})$ (resp. in $\Gamma(\tilde{V}^c)$). This replacement does not change the left hand side of \eqref{ineqV}. That these replacements are also elements of $\widetilde{\sigma}_m$ follows from the fact that $\sigma$ satisfies Part~(1) of Definition~\ref{finejacstab}. This concludes our proof.
   \end{proof}
We conclude by proving the following, which was used in the proof of the above.
\begin{lemma} Let ${\bf \underline{\tilde{e}}}$ and ${\bf \underline{\tilde{e}}}'$ be elements of $\tilde{\sigma}_m$. If  the inequality
\begin{equation} \label{ineq}
\left| \sum_{v \in W} {\bf \underline{\tilde{e}}}(v)-{\bf \underline{\tilde{e}}}'(v) \right|< \left|\Edges(W,W^c)\right|
\end{equation}
holds for all  $W \subset \Verts(\widetilde{\Gamma}_m)$ such that the induced subgraphs $\Gamma(W)$ and $\Gamma(W^c)$ are connected, then ${\bf \underline{\tilde{e}}}$ and ${\bf \underline{\tilde{e}}}'$ belong to different $\tw(\widetilde{\Gamma}_m)$-orbits. \label{blah} \end{lemma}

\begin{proof}
    Assume for a contradiction that there exists $f \colon \Verts(\widetilde{\Gamma}_m) \to \Z$ such that
    \[
    {\bf \underline{\tilde{e}}'}= 
    {\bf \underline{\tilde{e}}}+ \sum_{v \in \Verts(\widetilde{\Gamma}_m)} f(v) \tw_{\widetilde{\Gamma}_m, v}.
    \]
    Defining $W$ to be a subset of $\{w: \ f(w)=\min_{v\in\Verts(\Gamma)}(f)\}$ 
    such that both induced subgraphs $\Gamma(W)$ and $\Gamma(W^c)$ 
    are connected (one such $W$ subset always exists),
    we then deduce
    \[
    \sum_{v \in W}{\bf \underline{\tilde{e}}'}(v)\geq 
    \sum_{v \in W} {\bf \underline{\tilde{e}}}(v)+ |E(W,W^c)|.
    \]
    This contradicts \eqref{ineq}, thus concluding our proof.
\end{proof}
\section{From stability assignments to smoothable fine compactified Jacobians}
In this section we show how to construct a fine compactified  Jacobian from a given  stability assignment. We define a sheaf to be \emph{stable} if its multidegree is as prescribed by the  stability assignment, both in the case of a single nodal curve and for families. Then we show that the moduli space of stable sheaves thus defined is a fine compactified Jacobian as introduced in Definitions~\ref{def:finecompjac}~and~\ref{def:familyfinecompjac}. The case of a curve in isolation  is Corollary~\ref{stab->fine}, and the more general case of families with smooth generic element is Theorem~\ref{cor: fromstabtojac}. The most difficult part is the proof of properness (Lemma~\ref{properoverdelta}).

Throughout this section, let $X$ be a nodal curve over some field extension $K$ of $k$, and let $\Gamma$ be its dual graph. 
\begin{definition} \label{defstable}
 Let $\sigma$ be a degree $d$  stability assignment for $\Gamma$ (see Definition~\ref{finejacstab}). We say that $[F] \in \Simp^d(X)$  is \textbf{stable with respect to  $\sigma$} if  $(\gammazero(F),\underline{\deg}(F))$ is in $\sigma$, where $\gammazero(F)=\Gamma \setminus \NF(F)$ is the (necessarily connected and spanning) subgraph of $\Gamma$ that is dual to the normalization of $X$  at the points (necessarily nodes) where $F$ fails to be locally free. 
	
	We define  $\fcj_\sigma\subseteq \Simp^d(X)$ as the subscheme of sheaves that are stable with respect to $\sigma$.
\end{definition}

Let now $\mathcal{X}/S$ be a family of nodal curves over a $k$-scheme $S$. For every geometric point $s$ of $S$, denote by $\Gamma_s$ the dual graph of the fiber $\mathcal{X}_s$.

\begin{definition} \label{defstablefamily}
 Let $\mathfrak{S}=\set{\sigma_s}_{s \in S}$ be a family of  degree $d$  stability assignments for $\mathcal{X}/S$ as in Definition~\ref{familyfinejacstab}. 
 
 If $F$ is a geometric point of $\Simp^d(\mathcal{X}/S)$ lying over $s\in S$, we say that $F$  is \textbf{stable with respect to  $\mathfrak{S}$} if $F$ is stable with respect to $\sigma_s$ (see Definition~\ref{defstable}).
	
	We define  $\overline{\mathcal{J}}^d_{\mathfrak{S}}\subseteq \Simp^d(\mathcal{X}/S)$ as the algebraic subspace whose geometric points are sheaves that are stable with respect to $\mathfrak{S}$.
\end{definition}

The main result of this section is the following:

\begin{theorem} \label{cor: fromstabtojac}
Assume that $S$ is irreducible with generic point $\theta$, and that the generic element of the family $\mathcal{X}_\theta/\theta$ is smooth. Let $\mathfrak{S}$ be a family of degree~$d$ stability assignments for $\mathcal{X}/S$. Then $\overline{\mathcal{J}}^d_{\mathfrak{S}}\subseteq \Simp^d(\mathcal{X}/S)$ is a degree~$d$ fine compactified Jacobian for the family $\mathcal{X}/S$.
\end{theorem}

We immediately deduce:
\begin{corollary} \label{stab->fine}
The moduli space $\fcj_\sigma \subseteq \Simp^d(X)$ of $\sigma$-stable sheaves is a smoothable degree~$d$ fine compactified Jacobian for $X$.
\end{corollary}

\begin{proof}
    Apply Theorem~\ref{cor: fromstabtojac} to a regular smoothing $\mathcal{X}/\Delta$ of $X$.
\end{proof}

The most delicate part of the proof of Theorem~\ref{cor: fromstabtojac} is the following result, which builds on the combinatorial preparation of the previous section.

\begin{lemma} With the same assumptions as in Theorem~\ref{cor: fromstabtojac},  the moduli space $\overline{\mathcal{J}}^d_{\mathfrak{S}}$ is proper over $S$. \label{properoverdelta}
\end{lemma}

\begin{proof}
By Definitions~\ref{finejacstab} and \ref{defstable}, the scheme $\overline{\mathcal{J}}^d_{\mathfrak{S}}$ is the union of finitely many locally closed strata and hence of finite type and quasi-separated over $S$.
To prove properness, we apply the refined valuative criterion \cite[\href{https://stacks.math.columbia.edu/tag/0H1Z}{Lemma~70.22.3}]{stacks}.

Given the spectrum $\Delta$ of a DVR  and a commutative diagram (of solid arrows)
\[\xymatrix{
{\eta=\Delta\setminus\{0\}}\ar[r]^{ f}\ar@{^{(}->}[d]& {\Pic^d(\mathcal{X}_\theta)}\ar@{^{(}->}[r]&{\overline{\mathcal{J}}^d_{\mathfrak{S}}}\ar[d]\\
\Delta\ar[rr]\ar@{.>}[rru]&&{S},
}
\] 
we need to prove that there exists exactly one dotted arrow that keeps the diagram commutative. (Note that $\Pic^d(\mathcal{X}_{\theta})$ is open and dense in $\overline{\mathcal{J}}^d_{\mathfrak{S}}$). 
 
 Let us recall that morphisms to moduli spaces of sheaves (line bundles, simple sheaves) correspond to sheaves on the corresponding family of curves, but this correspondence is a bijection only after identifying two such sheaves whenever they differ by the pullback of a line bundle from the base. In the remainder of this proof we will slightly abuse the notation and assume this identification. 


Let us denote by $\mathcal X_{\Delta}/\Delta$ the base change of $\mathcal X/S$ under $\Delta\rightarrow S$ and by 
$\overline{\mathcal{J}}^d_{\Delta}\subseteq\Simp^d(\mathcal X/\Delta)$ the base change  of $\overline{\mathcal{J}}^d_{\mathfrak{S}}$. Let $s$ be the image of $0 \in \Delta$ under $\Delta \to S$. Denoting by $X= \mathcal{X}_s=\mathcal{X}_{\Delta,0}$  the  fiber over $0 \in \Delta$ of the family  $\mathcal{X}_{\Delta}/\Delta$, by the hypothesis that $\mathcal{X}_{\theta}/\theta$ is smooth we deduce that $\mathcal{X}_{\Delta}/\Delta$ is a smoothing of $X$. Moreover, we have that $\overline{\mathcal{J}}^d_{\Delta}$ is the disjoint union of $\Pic^d(\mathcal{X}_{\Delta, \eta})=\Simp^d(\mathcal{X}_{\Delta, \eta})$ and of $\fcj_\sigma \subseteq \Simp^d(X)$, where $\sigma$ is the stability assignment for $\mathcal{X}_s$. 

The morphism $f\colon \Delta\setminus\{0\}\rightarrow \Pic^d(\mathcal{X}_\theta)$ corresponds
to a line bundle $\mathcal [\mathcal{L}_{\eta}] \in \Pic^d(\mathcal{X}_{\Delta, \eta}/\eta)$. 
We know from  \cite{esteves} that $\Simp^d(\mathcal X_\Delta/\Delta)$ satisfies the existence part of the valuative criterion of properness,
 hence $\mathcal L_{\eta}$ extends (possibly nonuniquely) to an element $[\mathcal L'] \in \Simp^d(\mathcal{X}_\Delta/\Delta)$. Our proof is concluded if we can show that $\mathcal L_{\eta}$ extends to a unique $[\mathcal L] \in\overline{\mathcal{J}}^d_{ \Delta}$. When the total space  $\mathcal{X}_{\Delta}$ of the smoothing $\mathcal{X}_\Delta/\Delta$ is regular, this follows immediately from Remark~\ref{twister}.

In general, each  node $e$ of the central fiber $X \subset \mathcal{X}_\Delta$ is a singularity of type $A_{m(e)}$ of the total space $\mathcal X_\Delta$, for some $m(e) \in \mathbf{N}$. Let  $g \colon \widetilde{\mathcal{X}} \to \mathcal{X}_\Delta$ be the resolution of singularities obtained by blowing up $m(e)$ times each singularity $e$ of $\mathcal X_\Delta$. The central fiber $\widetilde{X} \subset \widetilde{\mathcal{X}}$ is therefore obtained by replacing each node $e$ of the central fiber $X$ of $\mathcal{X}_\Delta$ with a rational bridge of length $m(e)$. 
In other words, the dual graph $\widetilde{\Gamma}$ of $\widetilde{X}$ is obtained by subdividing $m(e)$ times each edge of $\Gamma=\Gamma(X)$  by adding $m(e)$ additional genus~$0$ vertices. 

We have now achieved that  $\widetilde{\mathcal{X}}$ is regular. Moreover, by \cite[Proposition~5.5]{esteves-pacini}, for each $[F] \in \Simp^d(X)$ there exists $[L] \in \Pic^d(\widetilde{X})$ such that $g_*(L)=F$, and the multidegree ${\bf d}$ of $F$ and the multidegree $\widetilde{\bf d}$ of any such $L$  are related by Equation~\eqref{lift}, where $\gammazero:=\gammazero(F)$ is the spanning subgraph consisting of all nodes where $F$ is locally free.

In order to prove the existence and uniqueness of an extension of $\mathcal L_{\eta}$ to  $\mathcal{X}_\Delta/\Delta$ with a sheaf in $\overline{\mathcal{J}}^d_{\Delta}$, we prove  that $g^*(\mathcal L_{\eta})$ extends uniquely on $\widetilde{\mathcal{X}}/\Delta$ to a line bundle whose restriction to $\widetilde{X}$ has multidegree in $\widetilde{\sigma}$, where the elements of $\widetilde{\sigma}$ are defined by lifting elements of $\sigma$ to $\widetilde{\Gamma}$ as in Equation~\eqref{lift}. Note that taking the direct image via $g$ maps line bundles on $\widetilde{X}$ whose multidegree is in $\widetilde{\sigma}$ to rank~$1$, torsion free, simple sheaves on $X$ whose multidegree is in $\sigma(\gammazero)$ for $\gammazero \subseteq \Gamma$ the spanning subgraph whose edges correspond to the exceptional rational bridges in $\widetilde{\Gamma}$ where the multidegree of the line bundle is equal to zero on all vertices. 

By Part~\eqref{due} of Proposition~\ref{coroll: chipfiring}, we deduce that  $\widetilde{\sigma}$ is a minimal complete set of representatives for the action of the twister group $\tw(\widetilde{\Gamma})$ on $S^d(\widetilde{\Gamma})$.

We  conclude, as in the regular case, that  $g^*(\mathcal L_{\eta})$ extends to a unique line bundle $\widetilde{\mathcal L}$ on $\widetilde{\mathcal{X}}/\Delta$ whose multidegree on the central fiber $\widetilde{X}$ is an element of $\widetilde{\sigma}$. Therefore $g_*{\widetilde{\mathcal{L}}}$ is the unique extension of $\mathcal L_{\eta}$ whose multidegree on $X$ is an element of $\sigma$. This completes the proof.
\end{proof}

 We are now ready to complete the proof of Theorem~\ref{cor: fromstabtojac}.

\begin{proof}[Proof of Theorem~\ref{cor: fromstabtojac}]  As $\overline{\mathcal{J}}^d_{\mathfrak{S}}$ is by definition constructible, openness is shown by proving that $\overline{\mathcal{J}}^d_{\mathfrak{S}} \subseteq \Simp^d(\mathcal{X}/S)$ is stable under generalization (\cite[ \href{https://stacks.math.columbia.edu/tag/0060}{Lemma~5.19.10}]{stacks}).   That  $\overline{\mathcal{J}}^d_{\mathfrak{S}}$ is stable under generalization follows from the first requirement of Definition~\ref{finejacstab}, and from the requirement that the family of stability assignments is compatible for graph morphisms arising from \'etale specializations (as specified in Definition~\ref{familyfinejacstab}).

Properness is   Lemma~\ref{properoverdelta}.
\end{proof}

\section{From smoothable fine compactified Jacobians to stability assignments}
In this section we show that a smoothable fine compactified Jacobian defines a stability assignment in a natural way. 

We start by considering the case of a curve in isolation and prove that the stability assignment associated with a (not necessarily smoothable) fine compactified Jacobian satisfies the first condition of Definition~\ref{finejacstab}. This enables us to describe in Section~\ref{examplesoffcj} all fine compactified Jacobians for two classes of examples: vine curves and curves whose dual graph has topological genus $1$. As an application, we prove in Lemma~\ref{lemma:support} that the stability assignment of a fine compactified Jacobian of an arbitrary nodal curve is nonempty for every connected spanning subgraph of the dual graph $\Gamma$ of the curve (a result needed in the proof of our main results, Corollaries~\ref{associsstab} and \ref{bijection}).

Then we restrict ourselves to the case where the fine compactified Jacobian is smoothable and proceed to prove the second condition of Definition~\ref{finejacstab} for the associated stability assignment (Corollary~\ref{associsstab}). Finally, we consider the case of families (Proposition~\ref{compcontract}) and we conclude by proving our main result, Corollary~\ref{bijection}, which establishes that the operations described in this section and in the previous one are inverse to each other. 

Throughout this section, we will consider a fixed nodal curve $X$ over a field
 extension $K$ of $k$, with dual graph $\Gamma$, and denote by $\fcj \subset \operatorname{Simp}^{d}(X)$ a degree $d$ fine compactified Jacobian of $X$.

\subsection{Definition and first properties}
\begin{definition} \label{assocassign}
We define the \textbf{associated assignment} $\sigma_{\fcj}$ of ${\fcj}$ as
\[
\sigma_{\fcj}=\{(\gammazero(F),\underline{\deg}(F)):\;[F]\in\fcj\},
\]
where $\gammazero(F)=\Gamma\setminus\NF(F)$ is the connected spanning subgraph of $\Gamma$ obtained by removing the edges corresponding to the nodes where $F$ fails to be locally free.
\end{definition}

A key  point  is that, if a sheaf is an element of a given fine compactified Jacobian, then so are all other sheaves that fail to be locally free on the same set of nodes and that have the same multidegree:

\begin{lemma} \label{exists-forall}
	Let  $\fcj \subset \Simp^d(X)$ be a degree $d$ fine compactified Jacobian, and assume that $(\gammazero,\underline{\mathbf d}) \in \sigma_{\fcj}$. Then $\fcj$ contains all sheaves $[F]\in \Simp^d(X)$ with $\NF(F) =\Gamma\setminus\gammazero$ and $\underline{\deg}(F)=\underline{\mathbf d}$.

\end{lemma}

\begin{proof}
By possibly passing to the partial normalization of $X$ at the nodes in $\Gamma\setminus\gammazero$ we may assume that $\gammazero=\Gamma$. Thus we aim to prove that if $\mathcal{J}^{\underline{\mathbf d}}(X)\cap \fcj \neq \emptyset$, then  $\mathcal{J}^{\underline{\mathbf d}}(X) \subset \fcj$.

The moduli space of line bundles $\mathcal{J}^{\underline{\mathbf d}}(X)$  of multidegree $\underline{\mathbf d}$ is irreducible and $\fcj$ is open, so 
$\mathcal{J}^{\underline{\mathbf d}}(X) \cap \fcj \neq \emptyset$ implies that $\mathcal{J}^{\underline{\mathbf d}}(X)\cap \fcj$
is dense in $\mathcal{J}^{\underline{\mathbf d}}(X)$. For $[F]$ in $\mathcal{J}^{\underline{\mathbf d}}(X)$ we can then find a line bundle $L$ on  $X \times \Delta$, for some $\Delta$, spectrum of a DVR, whose generic element $L_{\eta}$ is in $\mathcal{J}^{\underline{\mathbf d}}(X) \cap \fcj$ and whose special fiber $L_0$ equals $F$.

By properness of $\fcj$ there exists a sheaf $L'$ on $X\times \Delta$ with the property that  $L'_{X \times \eta}=L_{X \times \eta}$ (here $\eta$ is the generic point of $\Delta$) and whose special fiber $L'_0$ is in  $\fcj$. Then $L' \otimes L^{-1}$ defines a morphism $\Delta \to \Simp^0(X)$ that maps the generic point $\eta$ to a closed point, hence it must be the constant morphism. We conclude, in particular, that $[F]=[L_0']$ and so $[F] \in \fcj$.

\end{proof}

As a consequence of Lemma~\ref{impliedbyopen}, the associated assignment to a fine compactified Jacobian satisfies the first condition of Definition~\ref{finejacstab}:
\begin{corollary} \label{cond1}
The associated assignment $\sigma_{\fcj}$ satisfies
  \begin{equation} \label{Eqn: ChipAdding}
(\gammazero,\underline{\mathbf d})\in\sigma_{\fcj} \Rightarrow (\gammazero\cup\{e\},\underline{\mathbf d}+\underline{\mathbf e}_{v_1}),(\gammazero\cup\{e\},\underline{\mathbf d}+\underline{\mathbf e}_{v_2})\in\sigma_{\fcj}, 
\end{equation}
for all spanning subgraphs $\gammazero\subseteq\Gamma$ and all edges $e$ of $\Edges(\Gamma)\setminus\Edges(\gammazero)$ with endpoints $v_1$ and $v_2$.
\end{corollary}
Note that the above statement remains valid (with the same proof) for an arbitrary \emph{open} subscheme of $\Simp^d(X)$, properness does not play any role here.

For every connected spanning subgraph $G$ of $\Gamma$, the pairs $(\gammazero,\underline{\mathbf d})\in\sigma$ with $\gammazero\subset G$ define a fine compactified Jacobian on a partial normalization of the curves $X$.
\begin{lemma} \label{pullbackfcj} Let $f \colon X' \to X$ be a partial normalization of $X$ at some nodes $e_1, \ldots, e_k \in \Edges(\Gamma)$ and let $j_f \colon \Simp^{d-k}(X') \to \Simp^d(X)$ be the morphism induced by taking the pushforward along $f$. If $\fcj$ is a fine compactified Jacobian of $X$ then any geometrically connected component of $j_f^{-1} (\fcj)$ is a fine compactified Jacobian of $X'$. Moreover, for every $\gammazero \subseteq \Gamma(X')$, we have  \begin{equation} \label{geomconn} \bigcup_{J'\subset j_f^{-1}(\fcj)}\sigma_{J'}(\gammazero)=\sigma_{\fcj}(\gammazero),\end{equation}
where the union is over all geometrically connected components $J'$ of $j_f^{-1}(\fcj)$.
\end{lemma}
Note that Lemma~\ref{pullbackfcj} does not exclude the possibility that $j_f^{-1} (\fcj)= \emptyset$. This possibility will be excluded later, in Lemma~\ref{lemma:support}. Moreover, we observe that the pullback $j_f^{-1} (\fcj)$ is not necessarily connected: an example is given by nonsmoothable fine compactified Jacobians of curves of arithmetic genus $1$, see the $\rho\geq 2$ case of \cite[Lemma~3.7]{paganitommasi}. 

\begin{proof}
     The fact that $j_f^{-1} (\fcj)$  is open in $\Simp^{d-k}(X')$ follows from the fact that openness is stable under base change. 

    Moreover, $j_f^{-1} (\fcj)$ is separated and of finite type, because $\fcj$ is. To complete the first part of the statement, it remains to show that $j_f^{-1} (\fcj)$ satisfies the existence part of the valuative criterion of properness.

    Let $\Delta \setminus \{0 \} \to j_f^{-1} (\fcj)$ be a morphism out of a DVR without its closed point, which corresponds to an element $\mathcal{F'}$ in $\Simp^{d-k}( X'\times \eta /\eta)$. 
    Because $\fcj$ satisfies the existence part of the valuative criterion of properness, there exists a family of sheaves $\tilde{\mathcal{F}}$ in $\Simp^{d-k}(X \times \Delta/\Delta)$ whose restriction to $\eta$ equals $f_* (\mathcal{F})$. Then the family $f^*(\tilde{\mathcal{F}})/\text{torsion}$ is an extension of the original family $\mathcal{F}'$, and with central fiber belonging to $j_f^{-1} (\fcj)$ (because $f^*(f_*(\tilde{\mathcal{F}}))/\text{torsion}=\tilde{\mathcal{F}}$ -- see \cite[Lemma~1.5]{alexeev-cjtm}). This concludes the first part.

    The second part, or Equation~\eqref{geomconn}, follows immediately from the definition of an associated assignment.
\end{proof}

\subsection{Examples}\label{examplesoffcj}
In this subsection, we consider fine compactified Jacobians for two classes of curves: vine curves (defined in Example~\ref{ex: vinecurves}) and curves whose dual graph $\Gamma$ has first Betti number $1$. 

In the case of vine curves,  Corollary~\ref{cond1} allows us to describe explicitly all fine compactified Jacobians and conclude that they are automatically smoothable.

\begin{example} \label{vinecurves-nonsm} Let $X$ be a vine curve of type $t$ (as defined in Example~\ref{ex: vinecurves}) over some algebraically closed field $K$.
We claim that any fine compactified Jacobian $\fcj$ of $X$  is of the form $\fcj = \fcj_\sigma(X)$ for some stability assignment $\sigma$ as in Example~\ref{vinestab}. In particular, by Corollary~\ref{stab->fine}, $\fcj$ is smoothable.

This is clear when $t=1$, as in that case each geometrically connected component of $\Simp^d(X)=\Pic^d(X)$ is proper.

For the case of arbitrary $t>1$, we work by induction on $t$, using only Part~(1) of Definition~\ref{finejacstab} and the claim of Lemma~\ref{exists-forall} (neither of which relies on the hypothesis of smoothability).

For $e \in \Edges(\Gamma(X))$, we let $X_e$ be the  partial normalization of $X$ at the node $e$ only.  In view of $t >1$ we have that $X_e$ is connected  
and therefore $\Simp^{d-1}(X_e) \neq \emptyset$. Let $j_e \colon \Simp^{d-1}(X_e) \to \Simp^{d}(X)$ be the morphism obtained by taking the pushforward along the partial normalization $X_e \to X$.  

We observe that there exists an $e \in \Edges(\Gamma(X))$ such that  $j_e^{-1}\left(\fcj\right) \neq \emptyset$,  for otherwise we would have that $\fcj$ is contained in $\Pic^d(X)$, so $\fcj$ could not be proper.
Fix one such edge $e$ and let $\fcj_e \subseteq j_e^{-1}\left(\fcj\right)$ be a connected component. Then $\fcj_e$ is a degree $d-1$ fine compactified Jacobian of $X_e$, and $X_e$ is a vine curve of type $t-1$. We  apply the induction hypothesis to deduce that there exists  some degree $d-1$ stability assignment $\sigma_e$ as described in Example~\ref{vinestab} such that $\fcj_e=\fcj_{\sigma_e}(X_e)$.

Let $\sigma$ be the stability assignment on $X$ of the type studied in Example~\ref{vinestab} that is minimal among those that respect Condition~(1) of Definition~\ref{finejacstab} and whose restriction to $X_e$ equals $\sigma_e$.  By Corollary~\ref{cond1} we have that $\fcj_e=\fcj_\sigma(X)$ is contained in $\fcj$. By Corollary~\ref{stab->fine}, $\fcj_\sigma(X)$ is open in $\Simp^d(X)$, it is proper and geometrically connected. As $\fcj$ also enjoys these three properties, we conclude that $\fcj=\fcj_\sigma(X)$.
\end{example}

Next, we discuss curves whose unlabelled dual graph has genus $1$ and we prove that the combinatorial description of the strata of fine compactified Jacobians of nodal curves of arithmetic genus $1$ given in \cite[Section 3]{paganitommasi} generalizes to the fine compactified Jacobians of an arbitrary nodal curve $X$ with dual graph $\Gamma$ with $b_1(\Gamma)=1$.

\begin{lemma}\label{lemma:genus1}
If $\fcj$ is a fine compactified Jacobian of a curve $X$ with $b_1(\Gamma)=1$, then $\sigma_{\fcj}(\gammazero) \neq \emptyset$ holds for all connected spanning subgraphs $\gammazero \subseteq \Gamma$.    \end{lemma}
\begin{proof}

We can assume that $\Gamma$ has no separating edges. Indeed, by \cite[Lemma~2.7]{paganitommasi}, the fine compactified Jacobian of a curve with a separating node is the product of  fine compactified Jacobians of the two curves obtained by normalizing that node.

We may also assume that $\Gamma$ does not contain any loops based at one vertex, as the claim is trivial in that case. Let us assume that $|\Verts(\Gamma)|=n\geq 2$. Because $b_1(\Gamma)=1$, we can imagine $\Gamma$ as a \lq necklace\rq of $n$ vertices each connected to the next by one edge, each carrying its own genus weighting. The only possible connected spanning subgraphs of $\Gamma$ in this case are $\Gamma$ itself, and all graphs obtained by removing any one edge from $\Gamma$. 


The generalized Jacobian $\mathcal{J}^{\mathbf 0}(X)$ of $X$ is a semi-abelian variety which is a $\mathbf G_m$-bundle over the product of the Jacobians of the irreducible components of $X$. As a consequence, the strata $\mathcal {J}^d_{(G,\underline{\mathbf d})}$ of $\Simp^d(X)$ are either open  (when $G=\Gamma$) or closed  (when $G$ is obtained by removing one edge $e$ from $\Gamma$). The open strata are all isomorphic to $\mathcal{J}^{\mathbf 0}(X)$, while the closed ones are isomorphic to the abelian variety given by the product of the degree $0$ Jacobians of the components of $X$. 

Let $\fcj$ be a fine compactified Jacobian of $X$. By openness, there exists some $\underline{\mathbf{d}}$ such that $\fcj$ contains the open stratum $\mathcal{J}^d_{(\Gamma,\underline{\mathbf d)}}$.

We have that $\fcj$ satisfies the valuative criterion for properness; thus  the present situation is  similar to \cite[Section~3]{paganitommasi}.
Consider any morphism $\Delta \setminus \{0 \} \to \mathcal J^d_{(\Gamma,\underline{\mathbf{d}})}$ (where $\Delta$ is the spectrum of a DVR), which is not the restriction of a morphism $\Delta\to \mathcal J^d_{(\Gamma,\underline{\mathbf{d}})}$. 
Then the argument in \emph{loc. cit.} can be adapted to show that $f\colon\Delta \setminus \{0 \} \to \mathcal J^d_{(\Gamma,\underline{\mathbf{d}})}$ extends to a morphism $\Delta \to \Simp^d(X)$ in $n$  different ways, each one of them hitting a different closed stratum $\mathcal J^d_{(\Gamma\setminus\{e\},\underline{\mathbf{d}'})}$ for some $\underline{\mathbf{d}'}$.

This implies that the closure of $\mathcal J^d_{(\Gamma,\underline{\mathbf{d}})}$ inside $\Simp^d(X)$  looks like a compactification of the semi-abelian variety which is made non-separated by the fact that the divisors over both the $0$-section and the $\infty$-section of the extension of the $\mathbf G_m$-bundle $\mathcal{J}^{\mathbf 0}(X)$ to a $\mathbf P^1$-bundle are present in $n$ copies. 

Moreover, the poset associated with the stratification of $\Simp^d(X)$ is isomorphic to the poset associated with the stratification of $\Simp^d(X')$ where $\Gamma(X)$ and $\Gamma(X')$ have the same underlying graph, but possibly a different genus weighting on each vertex. Hence the combinatorial characterization of the neighbourhoods of the points of $\mathcal J^d_{(\Gamma\setminus\{e\},\underline{\mathbf{d}})}$ is the same as in the genus $1$ case.
 This allows us to conclude that the geometrical considerations applied to the description of the fine compactified Jacobians of $X'$ in \cite[Section~3]{paganitommasi} apply in the same way to the fine compactified Jacobians of $X$. This establishes, in particular, that $\sigma_{\fcj}(\Gamma\setminus\{e\})$ is non-empty for every edge $e$.
    \end{proof}

\begin{remark}\label{ex: genus1}
The proof of Lemma~\ref{lemma:genus1} above actually shows that if $X'$ is an arbitrary curve whose dual graph $\Gamma'$ is isomorphic to $\Gamma$, with possibly a different choice of the genera labelling the vertices, then the assignment $\sigma_{\fcj}$ associated with $\fcj$ also describes a fine compactified Jacobian on $X'$. We can deduce from this that the combinatorial description of fine compactified Jacobians of nodal curves of arithmetic genus $1$ given in \cite[Section 3]{paganitommasi} generalizes to fine compactified Jacobians of arbitrary nodal curves with dual graph $\Gamma$ with $b_1(\Gamma)=1$.
 \end{remark}

In preparation of our main results later (Corollaries~\ref{associsstab}, \ref{bijection}), we now prove the following.

\begin{lemma}\label{lemma:support} \footnote{This result is proved independently, and with a different argument, in \cite[Theorem~2.16]{viviani}.}
Let  $\fcj \subset \Simp^d(X)$ be a degree $d$ fine compactified Jacobian. Then $\sigma_{\fcj}(\gammazero) \neq \emptyset$ holds for all  connected spanning subgraphs $\gammazero \subseteq \Gamma$.
\end{lemma}

In the course of the proof, we will use the following
\begin{lemma}\label{Texists}
Let  $\fcj \subset \Simp^d(X)$ be a degree $d$ fine compactified Jacobian. Then there exists a spanning tree $T$ of $\Gamma$ such that $\sigma(T)\neq\emptyset$.
\end{lemma}

\begin{proof} By Lemma~\ref{exists-forall}, we have that $\fcj$ is a union of strata of $\Simp^d(X)$. The strata of $\Simp^d(X)$ in \eqref{eq: strata} that are universally closed are the minimal strata, i.e. those that correspond to the case where $\gammazero$ in \eqref{eq: strata} is a spanning tree of $\Gamma$. Because $\fcj$ is universally closed, it must contain at least one minimal stratum of $\Simp^d(X)$. 
\end{proof}

\begin{proof}[Proof of Lemma~\ref{lemma:support}]

Let us recall from Lemma~\ref{Texists} that there exists a spanning tree $T \subseteq \Gamma$ such that $\sigma_{\fcj}(T) \neq \emptyset$.
    If $b_1(\Gamma)=0$, then $\Gamma=T$ and the proof is concluded. If $b_1(\Gamma)=1$, then the result follows from Lemma~\ref{lemma:genus1}. From now on, we assume $b_1(\Gamma) \geq 2$.

Let us assume by contradiction that there exists a connected spanning subgraph $\gammazero$ such that $\sigma_{\fcj}(\gammazero) = \emptyset$. 
Let us choose $\gammazero$ maximal for this property; note that $\gammazero$ is always different from $\Gamma$ as a consequence of the fact that $\fcj$ is open in $\Simp^d(X)$. After possibly replacing $\Gamma$ with a subgraph containing $\gammazero$ and $X$ with the partial normalization corresponding to this subgraph, and by Lemma~\ref{pullbackfcj}, 
we may assume that $\gammazero$ is obtained from $\Gamma$ by removing a single edge $e$.

If $T \subseteq \gammazero$, then a contradiction arises immediately from Corollary~\ref{cond1}. We can therefore assume $T \not\subset \gammazero$, in other words that $e \in \Edges(T) \setminus \Edges(\gammazero)$.

There exists $e' \in \Edges(\gammazero) \setminus \Edges(T)$ such that the graph $T'$, obtained from $T$ by adding the edge $e'$ and removing the edge $e$, is connected (hence a spanning tree) and contained in $\gammazero$. We set $\Gamma'$ to be the graph obtained by adding $e'$ to the edge set of $T$ (or equivalently by adding $e$ to the edge set of $T'$). Then we have that $b_1(\Gamma')=1$.

Let $X' \subset X$ be the subcurve whose dual graph equals $\Gamma'$, let $k:=\left|\Edges(\Gamma) \setminus \Edges(\Gamma')\right|$, and let $j \colon \Simp^{d-k}(X') \to \Simp^{d}(X)$ be the morphism obtained by taking the pushforward along the partial normalization $X' \to X$.  Because $\sigma_{\fcj}(T) \neq \emptyset$, and by Lemma~\ref{pullbackfcj}, there exists some  geometrically connected component $J'$ of $j^*(\fcj)$ that is a degree $d-k$ fine compactified Jacobian of $X'$. 

Because $b_1(\Gamma')=1$, from the last paragraph, we can then conclude by Lemma~\ref{lemma:genus1} that $\sigma(T') \neq \emptyset$ for \emph{all} spanning trees $T' \subseteq \Gamma'$. In particular, there exists a $T' \subseteq \gammazero$ such that $\sigma(T') \neq \emptyset$, thus by Equation~\eqref{geomconn} combined with Corollary~\ref{cond1} we conclude that $\sigma_{\fcj}(\gammazero) \neq \emptyset$, which contradicts the assumption $\sigma_{\fcj}(\gammazero) = \emptyset$ made at the beginning.
\end{proof}

\subsection{Stability assignments of smoothable fine compactified Jacobians}

We now go back to our main line of reasoning, and show that the  assignment associated to  a \emph{smoothable} fine compactified Jacobian also satisfies the second condition of Definition~\ref{finejacstab}.

\begin{proposition} \label{cond2}
	If  $\fcj \subseteq \Simp^d(X)$ is also \emph{smoothable}, then for all spanning subgraphs $\gammazero \subseteq \Gamma$, the subset $\sigma_{\fcj}(\gammazero)\subset S^d_{\Gamma}(\gammazero)$ is a minimal complete set of representatives for the action of the twister group $\tw(\gammazero)$ on $S^d_{\Gamma}(\gammazero)$.
\end{proposition}

\begin{proof}

We first fix some notation for our proof. Let $\mathcal{X}/\Delta$ be a regular smoothing of $X$ such that there exists an open $\fcjs_{\mathcal{X}/\Delta} \subseteq \Simp^d(\mathcal{X}/\Delta)$  and $\Delta$-proper subscheme, whose special fiber equals $\fcj$. Then consider the degree~$2$ base change $\Delta \to \Delta$ of $\mathcal{X}/\Delta$, giving a nonregular smoothing $\mathcal{X}'/ \Delta$ with $A_1$ singularities at all nodes of the special fiber $X$. By functoriality of the Picard functor, and by stability under base change of openness and properness, we have that the base change $\fcjs'_{\mathcal{X}'/\Delta}$ is also an open and $\Delta$-proper subscheme of $\Simp^d(\mathcal{X}'/\Delta)$. Finally, let  $f \colon \widetilde{\mathcal{X}} \to \mathcal{X}' $  be the blow up  at all singularities. The family $\widetilde{\mathcal{X}}$ is then a regular smoothing of the special fiber $\widetilde{X}$, the curve obtained from $X$ by replacing each node with an irreducible rational bridge.

We are now ready for the proof. First  observe that, by Corollary~\ref{cond1} and by Lemma~\ref{lemma:support} combined with Proposition~\ref{barmak}, the set $\sigma_{\overline{J}}$ satisfies the  hypothesis of Proposition~\ref{coroll: chipfiring} Part~\eqref{uno}. By applying \emph{loc. cit.}, it is enough to prove that the collection $\widetilde{\sigma}_{\fcj}$ obtained by lifting every element of $\sigma_{\fcj}$ via Equation~\ref{lift} is a minimal complete set of representatives for the action of $\tw({\Gamma(\widetilde{X}}))$.

Let ${\bf \vec{d}}$ be the lift via \eqref{lift} of some ${\mathbf d} \in S^d_\Gamma(\gammazero)$ for some $\gammazero \subseteq \Gamma$. There exists $[L] \in \pic^d(\widetilde{X})$ such that $\underline{\deg}(f_*(L)) \in S^d_{\Gamma}(\gammazero)$. By Hensel's lemma, $L$ extends to a family $[\mathcal{L}] \in \Pic^d(\widetilde{\mathcal{X}}/\Delta)$. Since $\overline{J}_{\mathcal{X}'/\Delta}$ is universally closed over $\Delta$,  there exists $\mathcal{F}'\in \Simp^d(\mathcal{X}'/\Delta)$ such that $\mathcal{F}'|_{\mathcal{X}'_\eta}= f_* \mathcal{L}_{\mathcal{X}'_\eta}$ and $[\mathcal{F}'|_X]  \in \fcj_{X'}$. We take $[\mathcal{L}'] \in \Pic^d(\widetilde{\mathcal{X}}/\Delta)$ such that $f_*(\mathcal{L}')= \mathcal{F}'$. Thus we have that $\left(\mathcal{L}'\otimes \mathcal{L}^{-1}\right)_{\widetilde{\mathcal{X}}_{\eta}}=\mathcal{O}_{\mathcal{X}_{\eta}}$. We let then ${\bf \vec{t}}= \underline{\deg}(\mathcal{L}'\otimes \mathcal{L}^{-1}|_{\widetilde{X}})$, and so ${\bf\vec{d}}+{\bf\vec{t}}= \underline{\deg} (\mathcal{L}'|_{\widetilde{X}})$ is in $\widetilde{\sigma}_{\fcj}$. This proves that $\widetilde{\sigma}_{\fcj}$ is a complete set of representatives.

To prove minimality, assume that there exist lifts ${\bf \vec{d}_1}, {\bf \vec{d}_2} \in \widetilde{\sigma}_{\fcj}$  such that ${\bf \vec{d}_2}={\bf \vec{d}_1}+ {\bf \vec{t}}$ for some ${\bf \vec{t}} \in \tw(\Gamma(\widetilde{X}))$.
By Hensel's lemma and by Lemma~\ref{exists-forall}, there exist $[\mathcal{L}_1], [\mathcal{L}_2] \in \pic^d(\widetilde{\mathcal{X}}/\Delta)$, coinciding on the generic fiber, and whose multidegrees on $\widetilde{X}$ equal ${\bf \vec{d}_1}, {\bf \vec{d}_2}$. The pushforwards $f_*(\mathcal{L}_1)$ and $f_*(\mathcal{L}_2)$ coincide on  $\mathcal{X}'_{\eta}$, and because $\overline{J}_{\mathcal{X}'/\Delta}$ is separated over $\Delta$, their central fibers must coincide, which implies that ${\bf \vec{t}}=0$. This proves minimality.
\end{proof}

\begin{remark}
    \label{neron}  In \cite{caporasoneron} Caporaso introduced the notion of ``being of N\'eron type'' for a degree~$d$ compactified Jacobian of a nodal curve $X$. In our notation, this can be restated 
    as the property that $\sigma_{\fcj}(\Gamma(X))$ is a minimal complete set of representatives for the action of $\tw(\Gamma(X))$ on $S^d_{\Gamma(X)}(\Gamma(X))$. 
    
    Proposition~\ref{cond2} proves, in particular, that all smoothable fine compactified Jacobians are of N\'eron type. Similar results were obtained  in \cite{kass} and \cite{meloviv} for fine compactified Jacobians obtained from some numerical polarizations (see Section~\ref{Sec: OS} for the notion of a numerical polarization). 
    
    We refer the reader to \cite{kass} for the definition of a N\'eron model and its relations with  compactified Jacobians.
\end{remark}

\begin{remark}
The smoothability assumption
in the statement of Proposition~\ref{cond2} is crucial. In \cite{paganitommasi} the authors give an example, for $X$ a nodal curve of genus $1$, of degree $0$ fine compactified Jacobians $\fcj \subset \Simp^0(X)$ whose collection of line bundle multidegrees contain an arbitrary number $r\geq 2$ of elements for each orbit of the action of $\tw(\Gamma)$ on $S^0_{\Gamma}(\Gamma)$. Such compactified Jacobians can always be extended to a universally closed family over a regular smoothing, but such extensions are \emph{separated}  if and only if $r$ equals $1$.  

\end{remark}

We conclude with the main result:

\begin{corollary} \label{associsstab}
    The associated assignment (Definition~\ref{assocassign}) to a smoothable degree~$d$ fine compactified Jacobian of $X$ is a degree~$d$ stability assignment for the dual graph $\Gamma(X)$ (as in Definition~\ref{finejacstab}).
\end{corollary}

\begin{proof}
This is obtained as a combination of Corollary~\ref{cond1} and  Proposition~\ref{cond2}.
    \end{proof}

We conclude by observing that the same result holds for families. 
\begin{proposition} \label{compcontract} Let $\mathcal{X}/S$ be a family of nodal curves over a base scheme (or Deligne--Mumford stack) $S$, and let $\fcjs=\fcjs_{\mathcal{X}/S}$ be a degree~$d$ fine compactified Jacobian for the family. Assume that for each geometric point $s$ of $S$ the fine compactified Jacobian $\fcj_s$ of the fiber $X_s$ over $s \in S$ is \emph{smoothable}. 

Then the collection $\Set{\sigma_{\fcj_s}}$ of  associated assignments of $\fcj_s$ for all geometric points $s$ of $S$ is  a family of degree~$d$ stability assignments (as in Definition~\ref{familyfinejacstab}). 
\end{proposition}\begin{proof}
the case where $S$ is a single geometric point is Corollary~\ref{associsstab}.  

To complete our proof we will show that the assignment of $\sigma_{\fcj_s}$ is compatible with all morphisms  $f \colon \Gamma(\mathcal{X}_{s}) \to \Gamma(\mathcal{X}_t)$  arising from \'etale specializations $t \rightsquigarrow s$. 

Assume that $F_{t}$ is a simple sheaf on the nodal curve $\mathcal{X}_{t}$ that specializes to $F_s$ on $\mathcal{X}_s$.  Suppose that the subcurve $\mathcal{X}_{t, 0} \subseteq \mathcal{X}_{t}$ generalizes the subcurve $\mathcal{X}_{s, 0} \subseteq \mathcal{X}_s$. By flatness and by continuity of the Euler characteristic, and because we are passing to the torsion-free quotients, the degrees are related by
\begin{equation} \label{condition}
	 \deg_{\mathcal{X}_{t, 0}}(F_{t})= 	\deg_{\mathcal{X}_{s, 0}}(F_{s}) +  n(F_s,f)
		\end{equation}
		where $n(F_s,f)$ is the number of nodes of $\mathcal{X}_{s,0}$ that are smoothened in $t$ and where $F_s$ fails to be locally free. (See Equation~\eqref{deg:subcurve}).  Formula~\eqref{condition} is precisely the condition of Equation~\eqref{comp-contractions}.
\end{proof}

\begin{remark} \label{autdual}
    Assume that the dual graph of the fibers of $\mathcal{X}/S$ 
    is constant in $S$, and let $S$ be irreducible with $\eta \in S$ its generic point. Then the family $\mathcal X/S$ induces a group homomorphism from the group of \'etale specializations of $\eta$ to itself, 
      which equals the Galois group $\operatorname{Gal}(k(\eta)^{\operatorname{sep}}, k(\eta)))$, to the automorphism group $\operatorname{Aut}(\Gamma(X_{k(\eta)^{\operatorname{sep}}}))$ of the dual graph of the generic fiber. Let $G$ be the image of this group homomorphism. In this case, Proposition~\ref{compcontract} implies that the associated assignment $\sigma_{\fcj_{k(\eta)}}$, which is defined as a collection of discrete data on the dual graph $\Gamma(X_{k(\eta)^{\operatorname{sep}}})=\Gamma(X_{\overline{k(\eta)}})$, is  invariant under the action of $G$.

In the particular case where $S$ is a stratum of $\overline{\mathcal{M}}_{g,n}$, i.e. $S$ is the Deligne--Mumford moduli stack of curves whose dual graph is isomorphic to a fixed graph $\Gamma \in G_{g,n}$, we have $G=\operatorname{Aut}(\Gamma)$.
\end{remark}

By combining Theorem~\ref{cor: fromstabtojac}/Corollary~\ref{stab->fine} with Proposition~\ref{compcontract} and Lemma~\ref{exists-forall}, we immediately obtain that the two operations of (1) taking the associated assignment to a fine compactified Jacobian, and (2) constructing the moduli space of stable sheaves associated with a given stability assignment, are inverses of each  other.
\begin{corollary} \label{bijection}
Let $\mathcal{X}/S$ be a family of nodal curves over an irreducible scheme (or Deligne--Mumford stack)  $S$, and assume that $\mathcal{X}_{\theta}/\theta$ is smooth for $\theta$ the generic point of $S$.

If $\fcjs \subseteq \Simp^d(\mathcal{X}/S)$ is a degree~$d$ fine compactified Jacobian and $\sigma_{\fcjs}$ is its associated assignment (Definition~\ref{assocassign}), then the moduli space  $\overline{\mathcal{J}}_{\sigma_{\fcjs}}$ of $\sigma_{\fcjs}$-stable sheaves (Definitions~\ref{defstable} and~\ref{defstablefamily}) equals $\fcjs$.

 Conversely, if $\tau$ is a family of degree~$d$ stability assignments, and $\fcjs_\tau$ is the moduli space of $\tau$-stable sheaves, then the associated assignment $\sigma_{\fcjs_\tau}$ equals $\tau$.
\end{corollary}

\section{Numerical polarizations}

\label{Sec: OS}
An example of a (smoothable fine compactified Jacobian) stability assignment on a nodal curve $X/K$ (as in Definition~\ref{finejacstab}) comes from numerical polarizations, introduced by Oda--Seshadri in \cite{oda79}. (In fact, the Oda--Seshadri formalism permits to also construct compactified Jacobians that are not necessarily \emph{fine} in the sense of Definition~\ref{def:finecompjac}). 

In this section we review the notion of numerical stability for a single curve (Definition~\ref{OS- stable}) and extend it to families (Definition~\ref{familyphistab}) following \cite{kp3}.

We let $\Gamma$ be the dual graph of $X$. 

\begin{definition} \label{OS- stable} Let $V^d(\Gamma) \subset \mathbf{R}^{\Verts(\Gamma)}$ be the sum-$d$ affine subspace. Let ${\phi} \in V^d(\Gamma)$, and let $\gammazero$ be a (not necessarily connected) spanning subgraph of $\Gamma$, with $E_0:= \Edges(\gammazero)$ and $E_0^c:= \Edges(\Gamma) \setminus E_0$. We say  that $\underline{\mathbf d} \in S^d_{\Gamma}(\gammazero)$ is {\bf ${\phi}$-semistable (resp. ${\phi}$-stable)} on $\gammazero$ when  the inequality

\begin{equation}  \label{stabineq}\left|\sum_{v \in V}( \underline{\bf{d}}(v) -\phi(v)) + \left|E_0^c \cap E(\Gamma(V)) \right|+\frac{\left|E_0^c \cap \operatorname{E}(V,V^c)\right|}{2}\right|\leq \frac{\left|E_0 \cap \operatorname{E}(V,V^c) \right|}{2} 
\end{equation}

(resp. $<$) is satisfied for all $\emptyset \neq V \subsetneq \Verts(\Gamma)$. (To keep the notation compact, in the inequality we have denoted the edge sets by $\operatorname{E}$ instead of $\Edges$). 

For $\phi \in V^d(\Gamma)$ we define the  (numerical, Oda--Seshadri) assignment associated with $\phi$, denoted $\stab_{\Gamma, \phi}$ by setting
\begin{equation} \label{phisemistable}
\stab_{\Gamma, {\phi}}(\gammazero):= \Set{ \underline{\mathbf d} \in S^d_{\Gamma}(\gammazero): \underline{\mathbf d} \text{ is } {\phi} \text{-semistable on } \gammazero } \subset S^d_{\Gamma}(\gammazero)
\end{equation}
for all spanning subgraphs $\gammazero \subseteq \Gamma$.

We define  ${\phi} \in V^d(\Gamma)$  to be {\bf nondegenerate} when for every spanning subgraph $\gammazero\subseteq \Gamma$, all elements of  $\stab_{\Gamma, {\phi}}(\gammazero)$ are ${\phi}$-stable. 
\end{definition}

Elements of $V^d(\Gamma)$ are called {\bf numerical polarizations}.

\begin{remark} If the spanning subgraph $\gammazero \subseteq \Gamma$ is not connected, then $\Edges(\Gamma) \setminus \Edges(\gammazero)$ is a collection of edges that disconnects $\Gamma$. If we take for $V$ a subset of $\Verts(\Gamma)$ such that the induced subgraph $\Gamma(V)$ is a connected component of $\gammazero$, the right-hand side of Inequality~\eqref{stabineq}  equals zero. Therefore, if ${\phi}$ is nondegenerate, we always have $\stab_{\Gamma, {\phi}} (\gammazero) = \emptyset$ whenever $\gammazero$ is not connected.

We conclude that, for nondegenerate ${\phi}$'s, the disconnected spanning subgraphs $\gammazero$ of $\Gamma$ do not carry any additional information and they can be disregarded, as we have done in Definition~\ref{finejacstab}. 
\end{remark}

Every nondegenerate numerical polarization gives a stability assignment:

\begin{proposition} \label{OSisfine} Let $X$ be a nodal curve and let ${\phi} \in V^d(\Gamma(X))$ be nondegenerate. Then the numerical assignment $\stab_{\Gamma, {\phi}}$ defined in \ref{OS- stable} is a degree $d$ stability assignment (as  in \ref{finejacstab}) and the moduli space $\fcj_{\stab_{\Gamma, {\phi}}}(X)$ of  sheaves that are stable with respect to $\stab_{\Gamma, {\phi}}$ (as defined in \ref{defstable}) is a degree~$d$ smoothable fine compactified Jacobian.
\end{proposition}
\begin{proof} The fact that the moduli space is a smoothable fine compactified Jacobian is \cite[Proposition~2.9]{paganitommasi}. Thus $\stab_{\Gamma, {\phi}}$ defines a stability assignment by Corollary~\ref{associsstab}.
\end{proof}

This leads to the following natural question:
 \begin{question} \label{question} Let $\sigma_\Gamma$ be a degree $d$  stability assignment. Does there exist a ${\phi} \in V^d(\Gamma)$ such that $\sigma_{\Gamma} =  \stab_{\Gamma, {\phi}}$? (If one such ${\phi}$ exists, it is necessarily nondegenerate).
\end{question}
Below we give three classes of  examples where we can answer Question~\ref{question} in the positive. 
\footnote{As mentioned in the introduction, after a first draft of this paper appeared in the arXiv repository, Filippo Viviani constructed in \cite[Example~1.27]{viviani} an example of a stability assignment on a nodal curve of genus $3$ that is not induced by a numerical stability condition. The example in loc.cit. also uses the combinatorics of \cite[Example 6.15]{paganitommasi}.}

 \begin{example} \label{irred-is-OS} (Irreducible curves). If $X$ is irreducible, then there is a unique stability assignment $\sigma$  (see Example~ \ref{ex: stab-irred})  and we have $\sigma=\sigma_{\phi}$ for $\phi$ the only element of $V^d(X)=\{d\}$.
\end{example}
\begin{example} \label{dollar} (Vine curves of type $t$). Let $\Gamma$ consist of $2$ vertices $v_1, v_2$ connected by $t$ edges (and no loops). The complexity of $\Gamma$ equals $t$.
With the notation as in Example~\ref{vinestab}, it is straightforward to check that a stability assignment $\sigma_{\Gamma}$ determined by some $\lambda \in \Z$ as described in \emph{loc. cit.} equals $\stab_{\Gamma, \phi_\Gamma}$ for \[\phi_\Gamma(v_1, v_2):= \left(\lambda- \frac{t-1}{2}, d -\lambda + \frac{t-1}{2} \right),\]
and that $\phi_\Gamma \in V^d(\Gamma)$ is nondegenerate. 
\end{example}

\begin{example} \label{genus1} (Curves whose dual graph has genus $1$). Let $X$ be such that $b_1(\Gamma(X))=1$. A degree~$d$ stability assignment $\sigma$ on $\Gamma(X)$ is the same datum as a stability assignment on $\Gamma'$, where $\Gamma'$ is obtained from $\Gamma(X)$ by setting the genus of each vertex to $0$. Let $X'$ be a curve whose dual graph is $\Gamma'$. Then by Corollary~\ref{stab->fine} the scheme $\overline{J}_{\sigma}(X')$ is a smoothable fine compactified Jacobian. Smoothable fine compactified Jacobians on $X'$ have been classified in \cite{paganitommasi}: they all arise as $\overline{J}_{\sigma_\phi}(X')$ for some $\phi \in V^d(\Gamma')=V^d(\Gamma(X))$. Thus, by Corollary~\ref{associsstab} we have that $\sigma$ is also of that form.
\end{example}

\begin{example} \label{ibd} (Integral Break Divisors).  Let $\Gamma$ be a (not necessarily stable) graph and assume $d=g(\Gamma)$. Define the stability assignment $\sigma_{\operatorname{IBD}}$ by \[\sigma_{\operatorname{IBD}}(\gammazero)=\{ \text{integral break divisors for } \gammazero \},\] for all connected spanning subgraphs $\gammazero \subseteq \Gamma$. 
Recall that $\underline{\bf d} \in S^{g(\Gamma)}_{\Gamma}(\gammazero)$ is an integral break divisor for $\gammazero$ if and only if it is of the form
\[
\underline{\bf d} (v) =g(v) + \sum_{e \in \Edges(\gammazero) \setminus \Edges(T)}  \underline{\bf e}_{t(e)}(v)
\]
for some choice of a spanning tree $T \subseteq \gammazero$ and of an orientation $t \colon \Edges(\gammazero) \setminus \Edges(T) \to \Verts(\gammazero)$ (for $w \in \Verts(\gammazero)$ we denote by $\underline{\bf e}_w$ the function that is $1$ on the vertex $w$ and $0$ elsewhere).   

We claim that $\sigma_{\operatorname{IBD}}$ is indeed a degree $d=g(\Gamma)$ stability assignment. Condition~(1) of Definition~\ref{finejacstab} follows directly from the definition of an integral break divisor. The second axiom follows immediately from \cite[Theorem~1.3]{abkm} (earlier proved in \cite{mz}).

When $\Gamma$ is stable, it follows from \cite[Lemma~5.1.5]{cps} that $\sigma_{\operatorname{IBD}}= \sigma_{\Gamma, \phi_{\operatorname{can}}^{\Gamma}}$, for 
\begin{equation} \label{phican}
\phi_{\operatorname{can}}^{\Gamma}(v):= \frac{g(\Gamma) }{2 g(\Gamma) -2} \cdot \left(2g(v)-2+\operatorname{val}_{\Gamma}(v)\right)
\end{equation}
where $\operatorname{val}_{\Gamma}(v)$ is the valence of the vertex $v$ in $\Gamma$. (This follows from \emph{loc. cit.} as $\phi^{\Gamma}_{\operatorname{can}}$ is induced from the canonical divisor, which is ample when $\Gamma$ is stable).

When $\Gamma$ is not necessarily stable, we define  
\begin{equation} \label{phiIBD}
\phi_{\operatorname{IBD}}(v):= \frac{g(\Gamma) + |\Verts(\Gamma)|}{2 \left( g(\Gamma) + |\Verts(\Gamma)|\right) -2} \cdot (2g(v)+\operatorname{val}_{\Gamma}(v)) -1.
\end{equation}
We claim that $\sigma_{\operatorname{IBD}}= \sigma_{\Gamma, \phi_{\operatorname{IBD}}}$.

Let $\Gamma'$ be the graph obtained from $\Gamma$ by increasing the genus of each vertex by $1$. Then the integral break divisors of $\Gamma'$ are the integral break divisors on $\Gamma$ increased by $1$ on each vertex and $\Gamma'$ is stable. Therefore the integral break divisors on $\Gamma$ are stable for the stability assignment $\phi_{\operatorname{can}}^{\Gamma'}-1$, which equals $\phi_{\operatorname{IBD}}$. This proves our claim.

\end{example}

We are now ready to define families of numerical polarizations, similarly to what was done in Definition~\ref{familyfinejacstab} for families of stability assignments. We shall see that there are two ways of doing so, and that they are not equivalent.

\begin{definition} \label{def:morph}
Let $f \colon \Gamma \to \Gamma'$ be a morphism of stable graphs. We say that $\phi \in V^d(\Gamma)$ is $f$-compatible with $\phi' \in V^d(\Gamma')$ if
\[
\phi'(w)= \sum_{f(v)=w} \phi(v).
\]
\end{definition}

\begin{definition} \label{familyphistab}
Let $\mathcal{X}/S$ be a family of nodal curves, and let $\Phi=(\phi_s \in V^d(\Gamma(\mathcal{X}_s)) )_{s \in S}$ be a collection of numerical polarizations, one for each geometric point of $S$. Assume also that $\Phi$ is \textbf{nondegenerate}, by which we mean that so is each of its coordinates $\phi_s$ for $s \in S$ (as per Definition~\ref{OS- stable}).

The collection $\Phi$ is \textbf{weakly compatible} if the collection $\sigma_{\Phi}:=(\sigma_{\Gamma(X_s), \Phi(\Gamma(X_s))})_{s \in S}$ of degree~$d$ stability assignments defined via Proposition~\ref{OSisfine} is a family of degree $d$ stability assignments as prescribed by Definition~\ref{familyfinejacstab}.

The collection $\Phi$ is \textbf{strongly compatible} if it is $f$-compatible for all morphisms  $f \colon \Gamma(\mathcal{X}_{s}) \to \Gamma(\mathcal{X}_t)$ that arise from some \'etale specialization $t \rightsquigarrow s$ occurring on $S$.

\end{definition} 

For a strongly compatible collection, one can define the notion of a $\Phi$-stable sheaf for all fibers as in Definition~\ref{OS- stable}. We shall not do that, as this would be a repetition of our Definition~\ref{defstablefamily}. Instead, we observe the following:
\begin{corollary} \label{cor: familyphi}
    Let $\mathcal{X}/S$ be a family of nodal curves over an irreducible scheme $S$ with generic point $\theta$, and assume that the generic element $\mathcal{X}_\theta/\theta$ is smooth. If $\Phi$ is a nondegenerate and weakly compatible family of numerical polarizations, the moduli space $\overline{\mathcal{J}}_{\sigma_{\Phi}}$ of sheaves that are stable with respect to $\sigma_{\Phi}$ is a family of degree~$d$ fine compactified Jacobians.
\end{corollary}

\begin{proof}
    By combining Proposition~\ref{OSisfine} with the fact that $\Phi$ is nondegenerate and weakly compatible, we deduce that $\sigma_{\Phi}$ is a family of stability assignments for $\mathcal{X}/S$. The result follows then from Theorem~\ref{cor: fromstabtojac}.
\end{proof}

In \cite{kp3} the authors studied  the theory of universal Oda--Seshadri stability assignments for compactified Jacobians. In \emph{loc. cit.} they defined, for fixed  $(g,n)$ such that $2g-2+n>0$, the  space $V_{g,n}^d$ of {\bf universal numerical polarizations}, i.e. elements $\Phi=(\phi_s)_{s \in \Mb_{g,n}}$ that are strongly compatible for  morphisms $f \colon \Gamma_1 \to \Gamma_2$ between any two elements of ${G}_{g,n}$ (a skeleton of the category of stable $n$-pointed graphs of genus $g$). 

In \emph{loc. cit.} the authors also constructed, for each such nondegenerate $\Phi \in V_{g,n}^d$, a  compactified Jacobian $\Jb_{g,n}(\Phi)$. In the language of this paper, this can be rephrased as follows.
\begin{corollary} \label{stronglycomp}
If $\Phi\in V_{g,n}^d$ is a nondegenerate universal numerical stability condition, then $\Jb_{g,n}(\Phi) \subset \Simp^d(\Cb_{g,n}/\Mb_{g,n})$ is a fine compactified universal Jacobian.
\end{corollary}

\begin{remark} \label{Rem: phican}
    The space $V^d_{g,n}$ is never empty, as it always contains the \emph{universal canonical polarization}, defined as \(\Phi_{\operatorname{can}}^d := \frac{d}{2g-2}\; \underline{\deg}(\omega_{\pi}) \in V^d_{g,n}\). (When $d=g$ this coincides with the polarization defined in Equation~\eqref{phican}). It was observed in \cite[Remark~5.13]{kp3} that $\Phi_{\operatorname{can}}^d$  is nondegenerate precisely when $d-g+1$ and $2g-2$ are coprime. (In the $n=0$ case, this is already in \cite[p.~594]{caporaso}). 
\end{remark}

A natural question that we will address next is whether every fine compactified universal Jacobian arises as in Corollary~\ref{stronglycomp}, or if one can construct fine compactified universal Jacobian from nondegenerate weakly compatible numerical stability conditions that are not strongly compatible.

\begin{remark} \label{n>0}
It is clear that a strongly compatible nondegenerate family of numerical polarizations is also weakly compatible. The converse is not true. Even more: on some families $\mathcal{X}/S$ there exist nondegenerate families of numerical polarizations $\Phi=\left(\phi_s \in V(\Gamma_s)\right)$ that are weakly compatible, and such that  there exists no nondegenerate  $\Phi'=(\phi'_s)_{s \in S}$ that is strongly compatible and satisfies $\stab_{\Gamma_s, \phi'_s}=\stab_{\Gamma_s, \phi_s}$ for all  $s \in S$. 

Such examples are shown to exist in \cite[Section~6]{paganitommasi}, when $\mathcal{X}/S$ is the universal family $\overline{\mathcal{C}}_{1,n}/ \Mb_{1,n}$ and $n \geq 6$. In other words, for all $n\geq 6$ there are fine compactified universal Jacobians $\overline{\mathcal{J}}_{1,n}$ that are not of the form $\overline{\mathcal{J}}_{1,n}(\Phi)$ for any $\Phi \in V^d_{1,n}$.
\end{remark}

Our main result in the next section, Theorem~\ref{OS=fine},  settles in the affirmative the analogous question for the case of the universal family $\overline{\mathcal{C}}_{g}/ \Mb_{g}$ for all $g \geq 2$. More explicitly, when $n=0$, every fine compactified universal Jacobian $\overline{\mathcal{J}}_{g}$ arises as in Corollary~\ref{stronglycomp}, i.e. it is of the form $\overline{\mathcal{J}}_{g}(\Phi)$ for some  $\Phi \in V_g^d=V_{g,0}^d$.

\section{Classification of universal stability assignments}
\label{sec: univ}

The main result in this section is a classification of  \emph{universal}  stability assignments, that is, the case where the family $\mathcal{X}/S$ is the universal family over $S= \Mb_g$, the moduli stack of stable curves of genus $g$ (and no marked points). As an intermediate step, we will also produce results that are valid for the case of the universal family over $S= \Mb_{g,n}$ for arbitrary $n$.

One way to generate  universal  stability assignments is  by means of strongly compatible (universal) numerical polarizations, see Definition~\ref{familyphistab} and Corollary~\ref{stronglycomp}. Our main result here is Theorem~\ref{OS=fine}, where we classify all  universal  stability assignments with $n=0$ and for every genus, by proving that they all arise from strongly compatible numerical polarizations, i.e. they all are of the form $\stab_{\Phi}$ for some $\Phi \in V_g^d$ (see Corollary~\ref{cor: familyphi} and Corollary~\ref{stronglycomp}). Combining with the fact, shown in Proposition~\ref{compcontract}, that universal stability assignments classify fine compactified universal Jacobians, we deduce in Corollary~\ref{maincoroll} that, in the absence of marked points, there are no  fine compactified universal Jacobians other than the classical ones constructed in the nineties by Caporaso, Pandharipande and Simpson.

 As discussed in Remark~\ref{n>0}, an analogous result does not hold, in general, when $n>0$.

Recall from Example~\ref{ex: vinecurves} that, for $t \in \mathbf{N}$, a {\bf vine curve of type $t$} is a curve with $2$ nonsingular irreducible components joined by $t$ nodes. These curves and their dual graphs will play an important role in this section.

We will break down our argument in $3$ parts.

\subsection{First part: universal stability assignments are uniquely determined by their restrictions to vine curves}

The main result of this subsection, summarized in the title, is Corollary~\ref{determinedon2} below. The result is obtained as a combination of Lemma~\ref{enoughontrees} and Lemma~\ref{determinedover2}.  The first  establishes that a stability assignment for a given curve is uniquely determined by its restriction to the spanning trees of the dual graph of that curve. The second one shows that an assignment of integers on all spanning trees of all elements of $G_{g,n}$  that is compatible with graph morphisms is uniquely determined by its values over the dual graphs of all vine curves. 

Let us fix a degree~$d$ universal stability assignment $\sigma$ of type $(g,n)$. Recall from Definition~\ref{familyfinejacstab} (and Remark~\ref{Rem: univ stab}) that this is the datum  of a  collection \[\stab=\Set{\stab_{\Gamma}}_{\Gamma \in {G}_{g,n}}.\] that is compatible with all graph morphisms in $G_{g,n}$.

By Condition (1) of Definition~\ref{finejacstab}, if $T \subseteq \Gamma$ is a spanning \emph{tree}, then $\stab_{\Gamma}(T)= \set{\underline{\mathbf d}_{T}}$ contains a \emph{unique} multidegree  $\underline{\mathbf d}_{T} \in S^d_\Gamma(T)$. 
The following lemma shows that a stability assignment is ``overdetermined'' by the collection of data obtained by extracting this unique value from all spanning trees. 

\begin{lemma} \label{enoughontrees} Let $\Gamma$ be a stable graph and let us fix, for all spanning trees $T$ of $\Gamma$, a multidegree  $\underline{\mathbf d}_{T} \in S^d_\Gamma(T)$. Then there exists \emph{at most one} degree $d$ stability assignment $\stab_{\Gamma}$ whose value  $\stab_{\Gamma}(T)$ equals $\Set{\underline{\mathbf d}_{T}}$ for all spanning trees $T$.
\end{lemma}
\begin{proof}
Let $\stab_{\Gamma}$ be a degree~$d$ stability assignment that satisfies the hypothesis. Inductively define another collection $M_{\Gamma}$ by defining $M_{\Gamma}(\gammazero) \subset S_{\Gamma}^d(\gammazero)$ for all spanning subgraphs $\gammazero$ of $\Gamma$. Starting from the assignments $M_{\Gamma}(T):=\Set{d_{T}}$ for all spanning trees $T \subseteq \gammazero$,  let $M_{\Gamma}$ be the minimal assignment that satisfies Condition~(1) of Definition~\ref{finejacstab}. 

Because $\stab_{\Gamma}$ satisfies Condition~(1) of Definition~\ref{finejacstab}, for all spanning subgraphs $\gammazero \subseteq \Gamma$ we have the inclusion $M_{\Gamma}(\gammazero) \subseteq \stab_{\Gamma}(\gammazero)$. The inequality  $|M_{\Gamma}(\gammazero)|\geq c(\gammazero)$ follows from Proposition~\ref{barmak}. By Condition~(2) of Definition~\ref{finejacstab} we also have the equality $|\stab_{\Gamma}(\gammazero)| = c(\gammazero)$. From this we conclude that $M_{\Gamma}(\gammazero)$ and  $\stab_{\Gamma}(\gammazero)$ must coincide.
\end{proof}

By applying the same idea as in \cite[Lemma~3.8]{kp2} and \cite[Lemma~3.9]{kp3}, we  now see how compatibility with graph morphisms propagates an assignment of  multidegrees on all  strata of vine curves (with the exception of the vine curves whose graph admits a symmetry that swaps the two vertices) to an assignment on all vertices of all  spanning trees of all graphs $\Gamma \in G_{g,n}$.
The reason for excluding the symmetric graphs is that compatibility under automorphisms forces  all  universal degree~$d$ stability assignments to have the same value on those graphs. 

Let $T_{g,n} \subseteq G_{g,n}$ be the collection of loopless graphs with $2$ vertices, and let $T'_{g,n} \subseteq T_{g,n}$ be the subset of graphs such that each automorphism fixes the two vertices. For each element $G \in T_{g,n}$, fix an ordering of its two vertices, i.e. $\Verts(G)=\{v_1^G, v_2^G\}$. Let $C_{g,n} \subseteq T_{g,n}$ and be the subset of graphs with exactly $1$ edge (equivalently, the dual graphs of vine curves with no separating nodes, equivalently the dual graphs of vine curves whose strata have codimension $1$ in $\overline{\mathcal{M}}_{g,n}$), and let $G_{g,n}^{\operatorname{NS}} \subseteq G_{g,n}$ be the subset of graphs without separating edges.

\begin{lemma} \label{determinedover2} For each function $\alpha \colon T'_{g,n} \to \Z$, there  exists a unique collection of assignments $\underline{\mathbf d}_{\Gamma, T} \in S^d_{\Gamma}(T)$ for each spanning tree $T \subseteq \Gamma$ of each element $\Gamma \in G_{g,n}$ such that we have 
\begin{equation}\label{compcond} \sum_{v \in \Verts(\Gamma) : f(v) = v_1^G} \underline{\mathbf d}_{\Gamma, T}(v) = \alpha(G)
\end{equation}
for all morphisms $f \colon \Gamma \to 
G$ where $G$ is in $T_{g,n}$.

Similarly, for each function $\beta \colon T'_{g,n}\setminus C_{g,n} \to \Z$, there  exists a unique collection of assignments $\underline{\mathbf d}_{\Gamma, T} \in S^d_{\Gamma}(T)$ for each spanning tree $T \subseteq \Gamma$ of each element $\Gamma \in G_{g,n}^{\operatorname{NS}}$ such that \eqref{compcond} is satisfied for all morphisms
$f \colon \Gamma \to 
G$ where $G$ is in $T_{g,n}$.
\end{lemma}

Note that Equation~\eqref{compcond} is the same as compatibility for graph morphisms defined in Definition~\ref{combo-family}, Equation~\eqref{comp-contractions}.

\begin{proof} We only prove the first statement and leave the second to the reader. 

The statement is an extension of the proof given in \cite[Lemma 3.8]{kp2}. The main point of the proof is that for a fixed spanning tree $T \subseteq \Gamma$, contracting all but one edge of $T$ induces a bijection from $\Edges(T)$ to the set of morphisms from $\Gamma$ to some graph $G$ isomorphic to an element of $T_{g,n}$. The isomorphism is unique when $G \in T'_{g,n}$, and there are two isomorphisms when $G \in T_{g,n} \setminus T'_{g,n}$.


Each $e \in \Edges(T)$ induces a morphism $\Gamma \to \Gamma_e$ obtained by contracting all edges $\Edges(T) \setminus \{ e\}$.  There are two cases, depending on how each automorphism of $\Gamma_e$ acts on $\Verts(\Gamma_e)$:  (1) If each  acts trivially then there exists a unique isomorphism from $\Gamma_e$ to an element of $T'_{g,n}$. (2) if some act nontrivially, there are two isomorphisms from $\Gamma_e$ to an element of $T_{g,n} \setminus T'_{g,n}$.

Now we view $\underline{\bf d}_{\Gamma, T}$ as a vector with $|\Verts{T}|=|\Verts{\Gamma}|$ unknown entries. In Case~(1) we obtain an affine linear constraint among these unknowns by Equation~\eqref{compcond}. In Case~(2) a linear constraint is given by compatibility for the extra automorphism, which imposes that the sum of the values of $\underline{\bf d}_{\Gamma, T}$ on the vertices on one side of $e$ equals the sum of the values on the vertices on the other side of $e$.
Altogether, this gives $\left|\Edges(T)\right|=\left|\Verts(\Gamma)\right|-1$ affine linear constraints on the $\left|\Verts(\Gamma)\right|$ different entries of $\underline{\mathbf d}_{\Gamma, T}$. One more affine linear constraint is given by the fact that $\underline{\bf d}_{\Gamma, T} \in S^d_{\Gamma}(T) \subset \Z^{\Verts(\Gamma)}$. All in all, we obtain an affine linear system of $|\operatorname{Vert}(T)|$ equations in $|\operatorname{Vert}(T)|$ unknowns, which one can check has invertible determinant over $\Z$, hence  a unique solution for $\underline{\mathbf d}_{\Gamma, T}$.
\end{proof}

Lemmas~\ref{enoughontrees} and \ref{determinedover2} allow us to regard every degree~$d$ universal stability assignment of type $(g,n)$ as an element of $\Z^{T'_{g,n}}$. However, not all elements of $\Z^{T'_{g,n}}$ give rise to a universal stability assignment. For $\Gamma \in G_{g,n}$ it can happen that the collection on all spanning trees $\{\underline{\bf d}_{\Gamma,T}\}_{T}$ obtained from some $\alpha\in \Z^{T'_{g,n}}$ as in Lemma~\ref{determinedover2} is not the restriction of any stability assignment on the graph $\Gamma$ (see Lemma~\ref{enoughontrees}). 

\begin{corollary} \label{determinedon2}
    Let $\sigma, \tau$ be degree~$d$ universal stability assignments of type $(g,n)$. If $\sigma_G=\tau_G$ for all $G \in T'_{g,n}$, then $\sigma=\tau$. If $\sigma_G = \tau_G$ for all $G \in T'_{g,n} \setminus C_{g,n}$, then $\sigma$ and $\tau$ coincide on all curves that have no separating nodes.
\end{corollary}
\begin{proof}
    Follows immediately from Lemmas~\ref{enoughontrees} and \ref{determinedover2}.
\end{proof}

\subsection{Second part: when $n=0$,  over each vine curve with at least $2$ nodes there is at most \texorpdfstring{$1$}{1} stability assignment that extends to a universal stability} In this subsection we fix $g \geq 2$ and $n=0$. The main result here is Corollary~\ref{twovertexgraph}. 

We first need some combinatorial preparation. Let  $\GSym_g$ be the trivalent graph with $2g-2$ vertices $v_1, \ldots, v_{2g-2}$ of genus $0$, where each vertex $v_i$ is connected to $v_{i-1}$, $v_{i+1}$  and $v_{i+g-1}$ (indices should be considered modulo $2g-2$).
For $i\in\Z/(2g-2)\Z$ and $j=1,\dots,g-1$, we shall denote by $e_i$ the edge joining $v_i$ and $v_{i+1}$ and by $e_j'$ the edge joining $v_j$ and $v_{j+g-1}$. (See Picture~\ref{fig:gsigmag}).
\begin{figure}
  \includegraphics[scale=0.2]{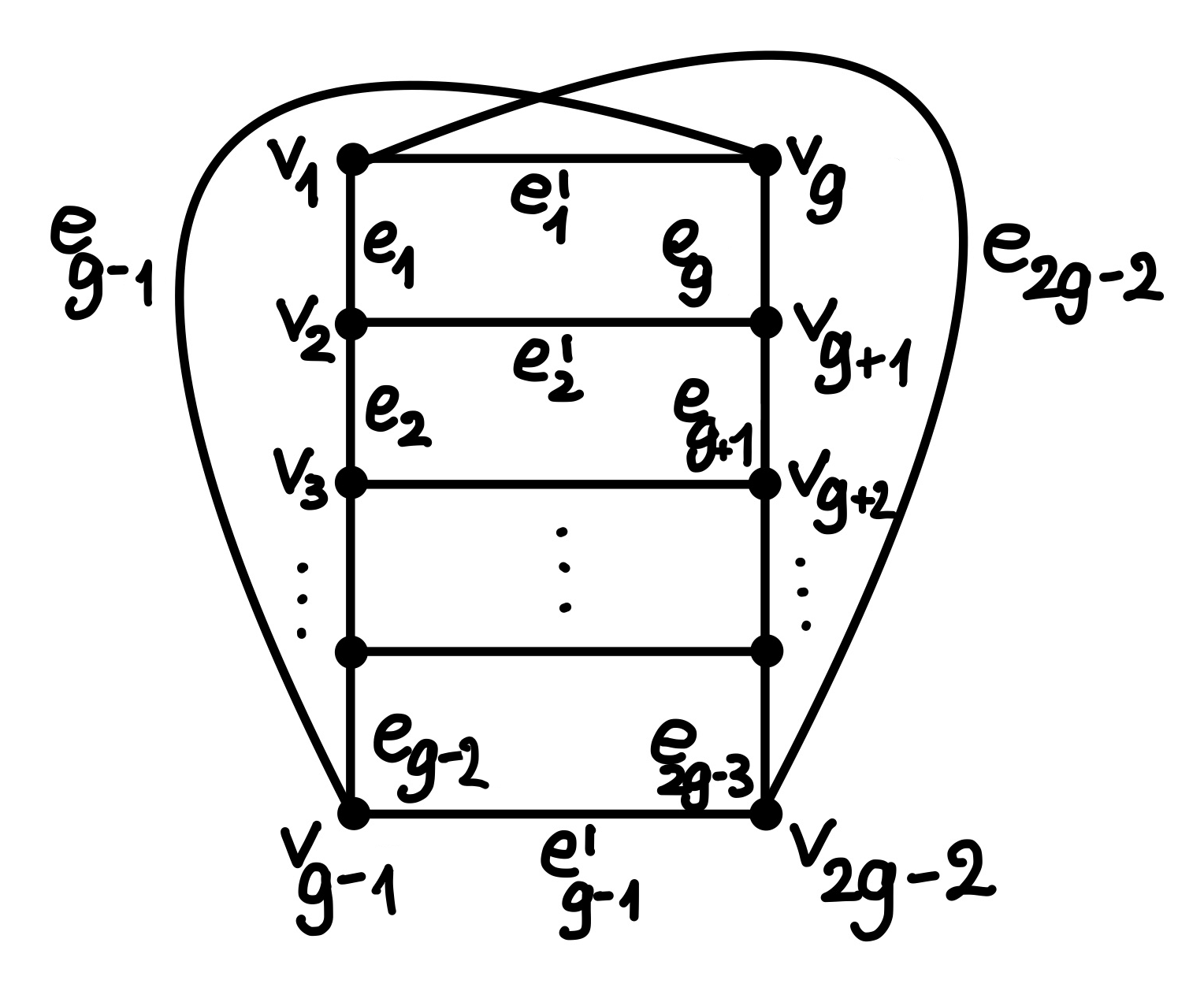}
  \caption{$\GSym_g$ \label{fig:gsigmag}}
  \end{figure}
Let $\Gamma_g \subset \GSym_g$ be the maximal $1$-cycle consisting of all edges of the form $e_i$.
\begin{lemma} \label{verysymmetric}  For every $d \in \Z$ there is at most one assignment \[\stab_{\GSym_g}(\Gamma_g)\subset S^d_{\GSym_g}(\Gamma_g)\] that satisfies the two conditions of Definition~ \ref{finejacstab}, and that extends to a degree $d$  universal stability assignment of type $(g,0)$. When such an assignment exists, the integers $d-g+1$ and $2g-2$ are necessarily coprime.
\end{lemma}
\begin{proof}
  Assume that $\stab_{\GSym_g}$ is the restriction to $\GSym_g$ of some  universal stability $\stab$. Since $\Gamma_g$ is a graph of genus $1$, we can apply the results from \cite[Section~3]{paganitommasi} to describe the assignment $\stab_{\GSym_g}(\Gamma_g)$. Namely, in genus $1$ it is known that all such assignments are induced by a polarization $\phi$ whose value $\phi(v_i)$ at each vertex $v_i$ is given by the average of the $2g-2$ admissible multidegrees in $\stab_{\GSym_g}(\Gamma_g)$. In particular, the polarization $\phi$ should be invariant under the automorphisms of $\GSym_g$ preserving $\Gamma_g$, such as the cyclic automorphism mapping $e_i$ to $e_{i+1}$ for all~$i$. From this we deduce that all $\phi(v_i)$ are equal, and since we have obtained $\Gamma_g$ from $\GSym_g$ by removing the $g-1$ edges $e_i'$, we have $\sum_{i\in\Z/(2g-2)\Z}\phi_i=d-g+1$. It follows that there exists at most $1$ assignment, and that this assignment should be the one induced by $\phi(v_i)=\frac{d-g+1}{2g-2}$ for all $i=1, \ldots, 2g-2$. Then the claim follows from the fact that this polarization is nondegenerate if and only if $d-g+1$ and $2g-2$ are coprime.
  \end{proof}

We can now prove the following.
\begin{corollary} \label{twovertexgraph} If $G\in T_{g,0} \setminus C_{g,0} \subseteq G_{g,0}$ is a loopless graph with $2$ vertices and at least $2$ edges, there exists at most one degree $d$  stability assignment for the graph $G$ that extends to a degree~$d$  universal  stability assignment. 
\end{corollary}

In the proof we will denote by $\Gamma(t,i,j)$ the object of $G_{g}$ that consists of two vertices of genus $i$ and $j$ connected by $t$ edges. (Note that $g=i+j+t-1$).
\begin{proof}
The claim is obtained by proving that the value of an assignment on vine curves of type $t \geq 2$ can be  obtained, using compatibility with graph morphisms (Definition~\ref{combo-family}), from the assignment calculated on $\Gamma_g \subset \GSym_g$ (described in Lemma~\ref{verysymmetric}). Recall that a degree~$d$ stability assignment over a graph with $2$ vertices has the same value on all spanning trees and is uniquely determined by this value (Example~\ref{vinestab}). 


First we prove the claim when $G = \Gamma(t, i,0)$ has one vertex of genus zero. Apply Lemma~\ref{verysymmetric} to deduce the uniqueness of an assignment $\stab_{\GSym_g}(\Gamma_g)$ that satisfies the two conditions of Definition~ \ref{finejacstab} and automorphism-invariance. Then apply to $\Gamma_g$ the first composition of contractions defined in \cite[Lemma 3.9]{kp3}.

The statement for an arbitrary graph $G= \Gamma(t, i, j)$ is obtained by applying the second set of contractions used in the proof of \cite[Lemma 3.9]{kp3}, to deduce the value of $\stab$ on $\Gamma(t, i, j)$ from the value of $\stab$ on $G=\Gamma( i-j+2, t+2j-2, 0)$. More precisely, consider the graph $G'$ with $4$ vertices $w_1,w_2,w_3,w_4$ defined in the proof of \cite[Lemma 3.9]{kp3}, and its spanning subgraph ${G}_0':=w_1-w_2-w_4-w_3-w_1$. Let $\alpha \in \Z$ be defined by the condition \[\stab_{\Gamma( i-j+2, t+2j-2, 0)}({G}_0)= \Set{(d-i+j -1-\alpha, \alpha) }\] for ${G}_0$ a spanning tree of $\Gamma( i-j+2, t+2j-2, 0)$. By combining the axioms of a degree~$d$ stability with compatibility with graph morphisms (Definition~\ref{combo-family}), we find out
\[
\stab_{G'}({G}_0')= \Set{(\beta+1, \beta, \alpha, \alpha), (\beta, \beta+1, \alpha, \alpha), (\beta, \beta, \alpha+1, \alpha), (\beta,\beta, \alpha, \alpha+1) }
\]
for the unique $\beta \in \Z$ such that $2\alpha+ 2 \beta = d+1 - t - i + j$ (observe that $d+1-t - i + j$ must be even since so is $d-g+1$, because by Lemma~\ref{verysymmetric} it is coprime with $2g-2$, and by using $g=t + i + j -1$). From Condition~(1) of Definition~\ref{finejacstab} we deduce that the unique assignment on the spanning tree ${G}_0'' \subset {G}_0'$ defined by $w_2- w_1 -w_3 - w_4$ equals
\[
\stab_{G'}({G}_0'') = \Set{(\beta, \beta, \alpha, \alpha) },
\]
and applying the second set of contractions used in the proof of \cite[Lemma 3.9]{kp3} we deduce the value of $\stab$ on  the graph $G= \Gamma(t, i, j)$.
\end{proof}

\subsection{Third part: Conclusion}

In this subsection, we fix $g \geq 2$ and $n=0$.

Note that if $\gcd(d-g+1,2g-2)=1$, the universal canonical numerical polarization  $\Phi_{\operatorname{can}}^d$  defined in Remark~\ref{Rem: phican} is nondegenerate, and so there is a  canonical universal stability assignment  $\stab_{\Phi^d_{\operatorname{can}}}$ of type $(g,0)$ defined via Definition~\ref{familyphistab}. More stability assignments can be obtained by modifying the stable bidegrees over the boundary divisors of $\overline{\mathcal{M}}_g$ having two components (the vine curves whose dual graph, borrowing from the notation of the previous subsection, has the form $G=\Gamma(1,i,g-i)$ for $1 \leq i<g/2$, the first vertex having genus $i$ and the second having genus $g-i$).

\begin{proposition} \label{rem: canonical}  Assume $\gcd(d-g+1,2g-2)=1$ and fix an integer $\alpha_i$ for each $1 \leq i <\lfloor \frac{g-1}{2} \rfloor$.  Then there exists a   stability assignment of the form $\stab_{\Phi}$ for $\Phi$ a nondegenerate element of $V_g^d$ (as in Definition~\ref{familyphistab}) such that $\Phi(G(1,i,g-i))=(\alpha_i, d-\alpha_i)$ for all $i$.  
\end{proposition}
\begin{proof} The existence of a nondegenerate universal stability assignment $\Phi$ that satisfies the given constraints follows immediately from \cite[Corollary~3.6, and Theorem~2, in the $n=0$ case]{kp3}. By Corollary~\ref{stronglycomp}, the assignment $\stab_{\Phi}$ defines a degree $d$  universal stability assignment.
\end{proof}

We are now in a position to conclude. 

\begin{theorem} \label{OS=fine} If a   degree~$d$  universal  stability assignment of type $(g,0)$ (Definition~\ref{familyfinejacstab} and Remark~\ref{Rem: univ stab}) exists, then $\gcd(d-g+1, 2g-2)=1$. Assuming the latter equality holds, let $\tau$ be one such stability assignment. Then there exists a unique stability assignment of the form $\sigma_{\Phi}$ for  $\Phi \in V_{g}^d$ nondegenerate such that $\tau$ equals $\stab_{\Phi}$ . This stability assignment coincides with the canonical stability assignment $\stab_{\Phi^d_{\operatorname{can}}}$ (see Remark~\ref{Rem: phican}) on all curves $[C] \in \overline{\mathcal{M}}_g$ with no separating nodes.
\end{theorem}
\begin{proof}
By  Lemma~\ref{verysymmetric} if $\gcd(d-g+1, 2g-2) \neq 1$ there exists no universal stability assignment. From now on we will assume $\gcd(d-g+1, 2g-2) = 1$.

We now prove the second statement. By the first part of Corollary~\ref{determinedon2},  proving the equality $\tau = \stab_{\Phi}$ for some nondegenerate $\Phi \in V_g^d$, is equivalent to proving the equality of the restrictions \begin{equation} \label{equality2} \tau_G = \stab_{G,\Phi(G)} \quad \text{for all loopless graphs } G \in {G}_{g,0} \text{ with } 2 \text{ vertices.}\end{equation} 
By Proposition~\ref{rem: canonical} there exists a stability assignment of the form $\stab_{\Phi}$ that coincides with $\tau$ when restricted to all loopless graphs $G$ with $2$ vertices and $1$ edge.  By the uniqueness statement in Corollary~\ref{twovertexgraph} we conclude that Condition~\eqref{equality2} holds, hence that $\tau = \stab_{\Phi}$. 

Our last statement follows from the second part of Corollary~\ref{determinedon2}.

\end{proof}

By combining Theorem~\ref{OS=fine} and Proposition~\ref{compcontract} (with $\mathcal{X}/S$ the universal family over $S=\Mb_g$), we immediately deduce the following classification result. 
\begin{corollary} \label{maincoroll} If $\overline{\mathcal{J}}_g$ is a degree~$d$ fine compactified universal Jacobian for $\Mb_g$, then $d$ satisfies $\gcd(d-g+1, 2g-2)=1$, and there exists $\Phi \in V_{g}^d$ such that $\overline{\mathcal{J}}_g = \Jb_{g} (\Phi)$.
\end{corollary}

We conclude by relating this result to the compactified universal Jacobians constructed in the nineties by Caporaso \cite{caporaso}, Pandharipande \cite{panda} and Simpson \cite{simpson}. By earlier work, this amounts to relating $\Jb(\Phi)$, for any $\Phi \in V_g^d$, to the canonical fine compactified universal Jacobian $\Jb(\Phi^d_{\text{can}})$.

\begin{remark}
    \label{maincoroll2} Let $\Phi \in V_g^d$ and assume $\gcd(d-g+1,2g-2)=1$. Then we claim that there exists a line bundle $M \in \Pic^0(\Cb_g)$ such that $\Phi=\Phi^d_{\operatorname{can}} + \underline{\deg}(M)$. In particular, translation by  $M$ induces an isomorphism $  \Jb_g(\Phi^d_{\operatorname{can}}) \to \Jb_g(\Phi)$ that commutes with the forgetful maps to $\Mb_g$.

    The claim follows from \cite[Section 6]{kp3}. In \emph{loc. cit.} the authors study the natural translation action of $\Pic^0(\Cb_{g,n})$ on $V^d_{g,n}$. This action induces an action on the collection of connected components of the nondegenerate locus, and by \cite[Lemma~6.15~(2)]{kp3}  this action is transitive when $n$ equals zero.
    \end{remark}

The relation between $\Jb_g(\Phi^d_{\operatorname{can}})$ and the compactified Jacobians of Caporaso \cite{caporaso}, Pandharipande \cite{panda} is discussed in \cite[Remark~5.14]{kp3}. The relation with Simpson's stability with respect to an ample line bundle is established in \cite[Corollary~4.3]{kp3} (see also \cite[Section~3]{kp2}).

\section{Final remarks and open questions}
\label{sec: final}

In Definitions~\ref{finejacstab} and \ref{familyfinejacstab} we  introduced the notion of families of degree~$d$ stability assignments for a single curve $X$ and for the universal family $\overline{\mathcal{C}}_{g,n}\to \Mb_{g,n}$.

For $X$ a nodal curve with dual graph $\Gamma$, let $\Sigma^d(\Gamma)$ be the set of all degree~$d$ stability assignments.  In Section~\ref{Sec: OS} we defined, following Oda--Seshadri \cite{oda79}, a stability space $V^d(\Gamma)$ with a subspace of degenerate elements (a union of hyperplanes). We define $\mathcal{P}^d(\Gamma)$ to be the set whose elements are the connected components (maximal dimensional polytopes) of the nondegenerate locus in $V^d(\Gamma)$. By Corollary~\ref{bijection} and Proposition~\ref{OSisfine}, there is an injection $\mathcal{P}^d(\Gamma) \to \Sigma^d(\Gamma)$, and it is natural to ask when this map is surjective. This is the same as Question~\ref{question}. the case where $g(X)=1$ (or even $b_1(\Gamma(X))=1$) was settled in the affirmative in \cite{paganitommasi}, see Example~\ref{genus1}.

Then let $\Sigma^d_{g,n}$ be the set of all degree~$d$  stability assignments of type $(g,n)$. The latter can be described as the inverse limit
\[\Sigma^d_{g,n}= \varprojlim_{G\in G_{g,n}} \Sigma^d(G)\]
where $G_{g,n}$ is the category of stable $n$-pointed dual graphs of genus $g$.

Similarly, for universal (and strongly compatible) numerical polarizations (see Section~\ref{Sec: OS}), in \cite{kp3} the authors  defined the affine space
\[
V^d_{g,n}:=\varprojlim_{G\in G_{g,n}}  V^d(G)
\]
(where morphisms are as in Definition~\ref{def:morph}). The latter is always nonempty, as it contains the canonical stability assignment. 

Let $\mathcal{P}^d_{g,n}$ be the set whose elements are the connected components of the complement of the degenerate stability assignments in $V_{g,n}^d$ (see Definition~\ref{familyphistab}). By Corollary~\ref{bijection} and Corollary~\ref{OSisfine}, we have a natural injection $\mathcal{P}^d_{g,n} \to \Sigma^d_{g,n}$, and one can ask under which assumptions the latter is surjective\footnote{While this paper was under review, this question has been answered completely in Fava's preprint \cite{fava}. In loc.cit., and building upon Viviani's results \cite{viviani}, the author also settles the question of what are the pairs $(g,n)$ such that  there exist fine compactified universal Jacobians of type $(g,n)$ whose fibres over $\overline{\mathcal{M}}_{g,n}$ are not all given by some numerical polarization.}. The genus $1$ case was settled in \cite{paganitommasi}: the map $\mathcal{P}^d_{1,n} \to \Sigma^d_{1,n}$ is a bijection if and only if $n \leq 5$ (see Remark~\ref{n>0}). The case without marked points is solved by Corollary~\ref{maincoroll}: the map $\mathcal{P}^d_g \to \Sigma^d_g$ is a bijection for all $g$.

Here are some natural future problems to address.
\begin{enumerate}
\item What is the analogue of Theorem~\ref{mainthm} for the case of  smoothable compactified Jacobians (not necessarily fine)? By this we mean a smoothable open substack of the moduli space of rank~$1$ torsion-free sheaves (not just the simple ones) that have a proper good moduli space.

\item Is there a natural stability space with walls, containing $V^d_{g,n}$, and a natural bijection from the set of its maximal-dimensional chambers to $\Sigma^d_{g,n}$?
\end{enumerate}

\bibliographystyle{alpha}
\bibliography{biblio-curves}

\newcommand{\etalchar}[1]{$^{#1}$}
\begin{thebibliography}{MMUV22}

\bibitem[ABKS14]{abkm}
Yang An, Matthew Baker, Greg Kuperberg, and Farbod Shokrieh.
\newblock Canonical representatives for divisor classes on tropical curves and
  the matrix-tree theorem.
\newblock {\em Forum Math. Sigma}, 2:25, 2014.
\newblock Id/No e24.

\bibitem[ACG11]{acg2}
Enrico Arbarello, Maurizio Cornalba, and Phillip~A. Griffiths.
\newblock {\em Geometry of algebraic curves. {V}olume {II}}, volume 268 of {\em
  Grundlehren der Mathematischen Wissenschaften [Fundamental Principles of
  Mathematical Sciences]}.
\newblock Springer, Heidelberg, 2011.
\newblock With a contribution by Joseph Daniel Harris.

\bibitem[AE12]{aminiesteves}
Omid Amini and Eduardo Esteves.
\newblock Voronoi tilings, toric arrangements and degenerations of line bundles
  i, 2012.
\newblock \href{https://arxiv.org/abs/2012.15620}{arXiv:2012.15620}.

\bibitem[AK80]{AK}
Allen~B. Altman and Steven~L. Kleiman.
\newblock Compactifying the {P}icard scheme.
\newblock {\em Adv. in Math.}, 35(1):50--112, 1980.

\bibitem[Ale04]{alexeev-cjtm}
Valery Alexeev.
\newblock Compactified {J}acobians and {T}orelli map.
\newblock {\em Publ. Res. Inst. Math. Sci.}, 40(4):1241--1265, 2004.

\bibitem[Bar17]{barmak}
Jonathan Barmak.
\newblock Private communication, 2017.

\bibitem[BHP{\etalchar{+}}23]{bhpss}
Younghan Bae, David Holmes, Rahul Pandharipande, Johannes Schmitt, and Rosa
  Schwarz.
\newblock Pixton's formula and {A}bel-{J}acobi theory on the {P}icard stack.
\newblock {\em Acta Math.}, 230(2):205--319, 2023.

\bibitem[BLM{\etalchar{+}}21]{bayer}
Arend Bayer, Mart\'{\i} Lahoz, Emanuele Macr\`{i}, Howard Nuer, Alexander
  Perry, and Paolo Stellari.
\newblock Stability conditions in families.
\newblock {\em Publ. Math. Inst. Hautes \'{E}tudes Sci.}, 133:157--325, 2021.

\bibitem[BLR90]{blr}
Siegfried Bosch, Werner L\"{u}tkebohmert, and Michel Raynaud.
\newblock {\em N\'{e}ron models}, volume~21 of {\em Ergebnisse der Mathematik
  und ihrer Grenzgebiete (3) [Results in Mathematics and Related Areas (3)]}.
\newblock Springer-Verlag, Berlin, 1990.

\bibitem[BMS06]{bms}
Simone Busonero, Margarida Melo, and Lidia Stoppino.
\newblock Combinatorial aspects of nodal curves.
\newblock {\em Matematiche}, 61(1):109--141, 2006.

\bibitem[Bri07]{bridgeland}
Tom Bridgeland.
\newblock Stability conditions on triangulated categories.
\newblock {\em Ann. of Math. (2)}, 166(2):317--345, 2007.

\bibitem[Cap94]{caporaso}
Lucia Caporaso.
\newblock A compactification of the universal {P}icard variety over the moduli
  space of stable curves.
\newblock {\em J. Amer. Math. Soc.}, 7(3):589--660, 1994.

\bibitem[Cap12]{caporasoneron}
Lucia Caporaso.
\newblock Compactified {J}acobians of {N}\'{e}ron type.
\newblock {\em Atti Accad. Naz. Lincei Rend. Lincei Mat. Appl.},
  23(2):213--227, 2012.

\bibitem[CCUW20]{CCUW}
Renzo Cavalieri, Melody Chan, Martin Ulirsch, and Jonathan Wise.
\newblock A moduli stack of tropical curves.
\newblock {\em Forum Math. Sigma}, 8:93, 2020.
\newblock Id/No e23.

\bibitem[CPS23]{cps}
Karl Christ, Sam Payne, and Tif Shen.
\newblock Compactified {J}acobians as {M}umford models.
\newblock {\em Trans. Amer. Math. Soc.}, 376(7):4605--4630, 2023.

\bibitem[EP16]{esteves-pacini}
Eduardo Esteves and Marco Pacini.
\newblock Semistable modifications of families of curves and compactified
  {J}acobians.
\newblock {\em Ark. Mat.}, 54(1):55--83, 2016.

\bibitem[Est01]{esteves}
Eduardo Esteves.
\newblock Compactifying the relative {J}acobian over families of reduced
  curves.
\newblock {\em Trans. Amer. Math. Soc.}, 353(8):3045--3095, 2001.

\bibitem[Fav24]{fava}
Marco Fava.
\newblock On a combinatorial classification of fine compactified universal
  {J}acobians, 2024.
\newblock \href{https://arxiv.org/abs/2403.17871}{arXiv:2403.17871}.

\bibitem[GST22]{gst}
Andreas Gross, Farbod Shokrieh, and Lilla T\'othm\'er\'esz.
\newblock Effective divisor classes on metric graphs.
\newblock {\em Math. Z.}, 302(2):663--685, 2022.

\bibitem[HKP18]{HKP}
David Holmes, Jesse~Leo Kass, and Nicola Pagani.
\newblock Extending the double ramification cycle using {J}acobians.
\newblock {\em Eur. J. Math.}, 4(3):1087--1099, 2018.

\bibitem[HMP{\etalchar{+}}22]{hmpps}
David Holmes, Samouil Molcho, Rahul Pandharipande, Aaron Pixton, and Johannes
  Schmitt.
\newblock Logarithmic double ramification cycles, 2022.
\newblock \href{https://arxiv.org/abs/2207.06778}{arXiv:2207.06778}.

\bibitem[Igu56]{igusa}
Jun-ichi Igusa.
\newblock Fibre systems of {J}acobian varieties.
\newblock {\em Amer. J. Math.}, 78:171--199, 1956.

\bibitem[Kas13]{kass}
Jesse~Leo Kass.
\newblock Two ways to degenerate the {J}acobian are the same.
\newblock {\em Algebra Number Theory}, 7(2):379--404, 2013.

\bibitem[KP17]{kp2}
Jesse~Leo Kass and Nicola Pagani.
\newblock Extensions of the universal theta divisor.
\newblock {\em Adv. Math.}, 321:221--268, 2017.

\bibitem[KP19]{kp3}
Jesse~Leo Kass and Nicola Pagani.
\newblock The stability space of compactified universal {J}acobians.
\newblock {\em Trans. Amer. Math. Soc.}, 372(7):4851--4887, 2019.

\bibitem[Mel19]{melouniversal}
Margarida Melo.
\newblock Universal compactified {J}acobians.
\newblock {\em Port. Math.}, 76(2):101--122, 2019.

\bibitem[MMUV22]{memoulvi}
Margarida Melo, Samouil Molcho, Martin Ulirsch, and Filippo Viviani.
\newblock {Tropicalization of the universal Jacobian}.
\newblock {\em {Épijournal de Géométrie Algébrique}}, {Volume 6}, August
  2022.

\bibitem[MRV17]{meravi}
Margarida Melo, Antonio Rapagnetta, and Filippo Viviani.
\newblock Fine compactified {J}acobians of reduced curves.
\newblock {\em Trans. Amer. Math. Soc.}, 369(8):5341--5402, 2017.

\bibitem[MV12a]{mv}
Margarida Melo and Filippo Viviani.
\newblock Fine compactified {J}acobians.
\newblock {\em Math. Nachr.}, 285(8-9):997--1031, 2012.

\bibitem[MV12b]{meloviv}
Margarida Melo and Filippo Viviani.
\newblock Fine compactified {J}acobians.
\newblock {\em Math. Nachr.}, 285(8-9):997--1031, 2012.

\bibitem[MZ08]{mz}
Grigory Mikhalkin and Ilia Zharkov.
\newblock Tropical curves, their {J}acobians and theta functions.
\newblock In {\em Curves and abelian varieties}, volume 465 of {\em Contemp.
  Math.}, pages 203--230. Amer. Math. Soc., Providence, RI, 2008.

\bibitem[OS79]{oda79}
Tadao Oda and C.~S. Seshadri.
\newblock Compactifications of the generalized {J}acobian variety.
\newblock {\em Trans. Amer. Math. Soc.}, 253:1--90, 1979.

\bibitem[Pan96]{panda}
Rahul Pandharipande.
\newblock A compactification over {$\overline {M}_g$} of the universal moduli
  space of slope-semistable vector bundles.
\newblock {\em J. Amer. Math. Soc.}, 9(2):425--471, 1996.

\bibitem[PT22]{paganitommasi}
Nicola Pagani and Orsola Tommasi.
\newblock Geometry of genus one fine compactified universal {J}acobians.
\newblock {\em Int. Math. Res. Not.}, 04 2022.
\newblock rnac094.

\bibitem[Ray70]{raynaud70}
M.~Raynaud.
\newblock Sp\'{e}cialisation du foncteur de {P}icard.
\newblock {\em Inst. Hautes \'{E}tudes Sci. Publ. Math.}, (38):27--76, 1970.

\bibitem[Sim94]{simpson}
C.~T. Simpson.
\newblock Moduli of representations of the fundamental group of a smooth
  projective variety. {I}.
\newblock {\em Inst. Hautes \'Etudes Sci. Publ. Math.}, 79:47--129, 1994.

\bibitem[{Sta}23]{stacks}
The {Stacks Project Authors}.
\newblock \textit{Stacks Project}.
\newblock \url{https://stacks.math.columbia.edu}, 2023.

\bibitem[Viv23]{viviani}
Filippo Viviani.
\newblock On the classification of fine compactified jacobians of nodal curves,
  2023.
\newblock \href{https://arxiv.org/abs/2310.20317}{arXiv:2310.20317}.

\bibitem[Yue18]{chiho}
Chi~Ho Yuen.
\newblock Geometric bijections of graphs and regular matroids, 2018.
\newblock PhD thesis, Georgia Institute of Technology, 2018.

\end{thebibliography}

\end{document}